\documentclass[twocolumn]{article}

\usepackage{amsmath}
\usepackage{amsthm}

\usepackage[dvips]{epsfig}    
\usepackage{float}
\usepackage{subfig}
\usepackage{tabularx}

\usepackage{amssymb}
\usepackage{latexsym}
\usepackage{tensor}
\usepackage{xfrac}

\usepackage{enumerate}

\usepackage[dvipsnames]{xcolor}

\makeatletter
\renewcommand*\env@matrix[1][*\c@MaxMatrixCols c]{%
  \hskip -\arraycolsep
  \let\@ifnextchar\new@ifnextchar
  \array{#1}}
\makeatother

\makeatletter
\def\mathcolor#1#{\@mathcolor{#1}}
\def\@mathcolor#1#2#3{%
  \protect\leavevmode
  \begingroup
    \color#1{#2}#3%
  \endgroup
}
\makeatother

\makeatletter
\def\hlt#1#{\@hlt{#1}}
\def\@hlt#1#2{%
  \protect\leavevmode
  \begingroup
      \color{BrickRed}{#1}#2%
  \endgroup
}
\makeatother

\font\BBbannan=msbm10 at 10pt
\newcommand{\beq}{\begin{equation}}
\newcommand{\eeq}{\end{equation}}
\newcommand{\beqr}{\begin{equation}\begin{array}{l}}
\newcommand{\eeqr}{\end{array}\end{equation}}
\newcommand{\beqa}{\begin{eqnarray}}
\newcommand{\eeqa}{\end{eqnarray}}

\newtheorem{theorem}{Theorem}[section]
\newtheorem{lemma}{Lemma}[section]
\newtheorem{remark}{Remark} [section]

\newtheorem{proposition}{Proposition} [section]

\newtheorem{corollary}{Corollary} [section]

\def\R{\mathbb{R}}
\def\F{\mathcal{F}}

\def\N{\mbox{\BBbannan N}}

\def\P{\mathsf{P}}
\def\p{\mathsf{p}}

\def\ds{\mathrm{d}s}
\def\dX{\mathrm{d}X}

\def\dY{\mathrm{d}Y}
\def\dW{\mathrm{d}W}
\def\Q{\mathsf{Q}}
\def\pR{\mathsf{R}}

\def\eps{\varepsilon}
\def\argmin{\operatornamewithlimits{arg\,min}}

\def\var{\operatorname{Var}}

\def\tr{\operatorname{tr}}

\def\ol{\overline}
\def\wh{\widehat}

\def\wt{\widetilde}
\def\E{\mathbf{E}}

\title{
Value Function Estimators for Feynman-Kac Forward-Backward SDEs 
    in Stochastic Optimal Control
} 

\author{Kelsey P. Hawkins, Ali Pakniyat, Panagiotis Tsiotras}

\date{}

\begin{document}

\maketitle

%


\begin{abstract}
    Two novel numerical estimators are proposed for solving forward-backward stochastic differential equations (FBSDEs) appearing in the Feynman-Kac representation of the value function in stochastic optimal control problems.
  In contrast to the current numerical approaches which are based on the discretization of the continuous-time FBSDE, we propose a converse approach, namely, we obtain a discrete-time approximation of the on-policy value function, and then we derive a discrete-time estimator that resembles the continuous-time counterpart.
  The proposed approach allows for the construction of higher accuracy estimators along with error analysis.
  The approach is applied to the policy improvement step in reinforcement learning. Numerical results and error analysis are demonstrated using (i) a scalar nonlinear stochastic optimal control problem and (ii) a four-dimensional linear quadratic regulator (LQR) problem.
  The proposed estimators show  significant improvement in terms of accuracy in both cases over Euler-Maruyama-based estimators used in competing approaches. 
  In the case of LQR problems, we demonstrate that our estimators result in near machine-precision level accuracy,  in contrast to previously proposed methods that can potentially diverge on the same problems.
\end{abstract}

\section{Introduction} 

Feynman-Kac representation theory and its associated forward-backward stochastic differential equations \mbox{(FBSDEs)} has been gaining traction as a framework to solve nonlinear stochastic optimal control problems, including problems with quadratic cost \cite{exarchos2018stochastic}, minimum-fuel ($L_1$-running cost) problems \cite{Exarchos2018}, differential games \cite{Exarchos2016,Exarchos2018a}, and reachability problems \cite{exarchos2018stochastic,mete2002stochastic}. Although FBSDE-based methods have seen growing attention in both the controls and robotics communities recently, much of the research originated in the mathematical finance community \cite{Bender2007,Longstaff2001,ma2007forward}. 
While initial results demonstrate promise in terms of flexibility and theoretical validity, numerical algorithms which leverage this theory have not yet matured. 
For even modest problems, state-of-the-art algorithms can be unstable, producing value function
    approximations which quickly diverge. 
Thus, producing more robust numerical methods is critical for the broader adoption of FBSDE methods for real-world tasks.

Numerical solution methods for Feynman-Kac-type \mbox{FBSDEs} broadly consist of two steps, a forward pass, 
    which generates Monte Carlo samples of the forward stochastic process, 
    and a backward pass, which iteratively approximates the value function backwards in time. 
Typically, FBSDE methods perform this approximation using a least-squares Monte Carlo (LSMC) scheme, 
    which implicitly solves the backward SDE using parametric function approximation \cite{Longstaff2001}. 
The approximate value function fit in the backward pass is then used to improve sampling 
    in an updated forward pass, leading to an iterative algorithm which, ideally, improves 
    the approximation till convergence.
Although FBSDE methods seem similar to differential dynamic programming (DDP) 
    techniques \cite{jacobson1970differential,theodorou2010stochastic,NIPS2007_3297},
    the approach is significantly different.
DDP methods require first and second order derivatives of the dynamics, and directly
    compute a quadratic approximation of the value function using constraints on the derivatives
    of the value function.
Comparatively, FBSDE LSMC only uses estimates of the value function at a distribution
    of states, using the derivative of the value function to improve the accuracy of
    the estimator.
FBSDE methods are more flexible in that they do not require 
    derivatives of the dynamics and can be used with models of the value function
    which are not necessarily quadratic.
Furthermore, for most DDP applications, a quadratic running cost with respect to the control 
    is required for appropriate regularization
    \cite[Section 2.2.3]{tassa2011theory}, 
    whereas the FBSDE method more easily accommodates non-quadratic running costs
    (e.g., of the class $L_1$ or zero-valued), lending to a variety of
    control applications \cite{Exarchos2018}.


    The underlying foundation of Feynman-Kac-based FBSDE algorithms  is the intrinsic relationship between the solution of a broad class of second-order parabolic or elliptic PDEs to the solution of FBSDEs (see, e.g., \cite[Chapter 7]{yong1999stochastic}), brought to prominence in \cite{Pardoux1990,peng1993backward,el1997backward}. Both Hamilton-Jacobi-Bellman (HJB) and Hamilton-Jacobi-Isaacs (HJI) second order PDEs, utilized for solving, respectively, stochastic optimal control and stochastic differential game problems, can thus be solved via FBSDE methods, even when the dynamics and costs are nonlinear and non-quadratic, respectively. This provides an alternative to the grid-based direct solution of PDEs, typically solved using finite-difference, finite-element, or level-set schemes, known for poor scaling in high dimensional state spaces ($n \geq 4$). 


In this work, we investigate the discrete-time approximation of the backward SDE in the context
    of solving for the value function in the backward pass in FBSDE methods.
Although  for some special cases analytic solutions of the backward SDEs over short intervals can be accommodated into the associated algorithms \cite{Longstaff2001}, 
    for many nonlinear problems analytic solutions are not available and numerical integration based on time-discretization is a necessity.
In  the currently available algorithms in the literature Euler-Maruyama approximations are employed for discretizing the continuous-time 
    FBSDEs \cite{exarchos2018stochastic}, to solve for an approximation of the
    continuous-time value function.
In this paper, 
instead of the direct application of the Euler-Maruyama approximation on the Feynman-Kac FBSDEs, we formulate a discrete time problem with the Euler-Maruyama approximation of the dynamics, costs, and value function, and then
 we
    derive discrete-time relationships using Taylor expansions which resemble their 
    continuous-time counterparts.
By doing so, we arrive at a set of alternative estimators for the value function.

The primary contributions of this paper are as follows:

\begin{itemize}
    \item Proposing a pair of alternative estimators for the value function used in the backward pass 
          of a Girsanov-drifted Feynman-Kac FBSDE numerical method.
    \item Characterizing the theoretical bias and variance of these estimators and show
          their theoretic superiority to previously proposed estimators.
    \item Numerically confirming the theoretical results on representative stochastic optimal control problems.
\end{itemize}

This paper expands upon the authors' prior work in \cite{CDCdtfbsde2021}, first by
    providing more details into how the proposed estimators are constructed.
Second, we provide detailed proofs for the stated theorems, especially the discrete-time
    version of Girsanov's theorem, which allows for the interpretation of the error.
In addition, we discuss how the methodology can be adapted to produce an approximate
    policy improvement.
Finally, in addition to a more detailed presentation of the scalar nonlinear example in \cite{CDCdtfbsde2021}, we  present results of experiments on a four-dimensional LQR problem,
    verifying our theoretical claims about the accuracy of the proposed estimators.

The structure of the paper is as follows.
In Section~\ref{sec:dtfksbsdes} we introduce the stochastic optimal control problem 
    we are interested in, as well as the continuous-time approach to solving 
    for an on-policy value function using drifted FBSDEs.
At the end of this section we describe a discrete-time method of approximating the backward SDE 
    which we will improve upon.
In Section~\ref{sec:fbdeqs} we introduce our proposed approach, beginning by 
    replacing the continuous-time problem with a discrete-time approximation.
We then use discrete-time relationships to arrive at estimators which resemble the estimators
derived from continuous-time theory.
We also provide an error analysis for the proposed estimators.
Next, in Section~\ref{sec:policyimprove}, we briefly show how a similar approach to derive
    the estimators can be used to approximate the Q-value function for policy improvement
    methods in reinforcement learning problems.
Finally, in Section~\ref{sec:expresults} we present results from two numerical experiments which
    confirm the error analysis and illustrate the benefits of our approach over previously proposed estimators.

\section{Continuous-Time Feynman-Kac FBSDEs} \label{sec:dtfksbsdes}

In this section we introduce the stochastic optimal control problem we are interested in,
    and show how its solution can be obtained as a pair of forward-backward stochastic
    differential equations (FBSDEs).
Further, we discuss how these continous-time FBSDEs can be approximated using the Euler-Maruyama method.

\subsection{Stochastic Optimal and On-Policy Value Functions} \label{sec:soc}

We start with a complete, filtered probability space ${(\Omega, \F, \{\F_t\}_{t \in [0,T]}, \Q)}$, 
    on which $W_s^\Q$ is an $n$-dimensional standard Brownian (Wiener) process with respect 
    to the probability measure $\Q$ and adapted to the filtration $\{\F_t\}_{t \in [0,T]}$. 
Consider a stochastic nonlinear system governed by the It\^{o}~differential equation
\begin{align}
    \dX_s &= f(s,X_s,u_s)\, \ds + \sigma(s,X_s) \, \dW^\Q_s \text{,} & X_0 &= x_0 \text{,}
    \label{eq:SOCdyn}
\end{align}
    where $X_s$ is a state process 
    taking values in $\R^n$, 
    $u_{[0,T]}$ is a measurable and adapted input process 
    taking values in the compact set 
    $U \subseteq \R^m$, 
    and $f: [0,T] \times \R^n \times U \rightarrow \R^n$,
    $\sigma: [0,T] \times \R^n \rightarrow \R^{n \times n}$ are the Markovian drift
    and diffusion functions, respectively.
The cost associated with a given control signal $u_{[t,T]}$ is
\begin{align}
    S_t(u_{[t,T]}) &:= \int^T_t \ell(s,X_s,u_s) \, \ds + g(X_T) \text{,} 
\end{align}
    where $\ell: [0,T] \times \R^n \times U \rightarrow \R_+$ is the running cost, 
    and \mbox{$g: \R^n \rightarrow \R_+$} is the terminal cost. 
We assume that 
    $f, \sigma, \ell, g$ are uniformly continuous and Lipschitz in $x$ for all ${t \in [0,T]}, {u \in U}$,
    and that $\sigma^{-1}$ exists and is uniformly bounded on its domain.

The stochastic optimal control (SOC) problem is to determine the optimal value function
\begin{align}
    V^*(t,x) &= \inf_{u_{[t,T]}} \big \{ \E_{{\Q}}[\, S_t(u_{[t,T]}) \, | X_t = x] \big \}
    \tag{SOC} \label{eq:SOC} \text{,} 
\end{align}
    (see \cite[Section~4.3]{yong1999stochastic}).
Given the previous assumptions on the dynamics and costs,
    $V^*$ is a unique viscosity solution, continuous on $[0,T] \times \R^n$,
    of the associated 
    Hamilton-Jacobi-Bellman PDE \cite[Chapter~4 Theorem~5.2; Theorem~6.1]{yong1999stochastic}.

The iterative approach to solving the optimal control problem is to successively
    improve approximations of the optimal policy and optimal value function $(\pi^*, V^*)$,
    refining an arbitrary policy $\mu$ and its associated on-policy value function $V^\mu$,
    which characterizes the cost-to-go under this policy.
Consider the space of admissible feedback policies, that is, measurable functions
    $\mu : [0,T] \times \R^n \rightarrow U$ for which there exists a weak SDE solution for
\begin{align}
    \dX_s &= f^\mu_s \, \ds + \sigma_s \, \dW^\Q_s, & X_0 &= x_0 \text{,}
    \label{eq:fsdeorig} 
\end{align}
    where  
    $f^\mu_s := f^\mu(s,X_s)$,
    $f^\mu := f(t,x,\mu(t,x))$, 
    and henceforth abbreviate $\ell$, and $\sigma$ similarly.
The on-policy value function $V^\mu$ is defined as 
\begin{align}
    \begin{aligned}
        V^\mu(t,x) &= \E_{{\Q}}[\, S^\mu_t \, | X_t = x] \text{,} \\
        S^\mu_t &:= \int^T_t \ell^\mu_s \, \ds + g(X_T) \text{,} 
    \end{aligned} \label{eq:onpolicyvf}
\end{align}
    with the process $X_s$ satisfying the forward SDE (FSDE) \eqref{eq:fsdeorig}, and
    its associated Hamilton-Jacobi PDE is
\begin{align}
    \begin{aligned}
        \partial_t V^\mu + \frac{1}{2} \tr[\sigma \sigma^\top \partial_{xx} V^\mu] 
        + (\partial_{x} V^\mu)^\top f^\mu + \ell^\mu \big |_{t,x} = 0 & \text{,} \\
        V^\mu(T,x) = g(x) \text{,} &
    \end{aligned} \label{eq:hjpde}
\end{align}
    for $(t,x) \in [0,T) \times \R^n$, where $\partial_t$ and $\partial_x$
    are the partial derivative operators with respect to 
    $t$ and $x$, and $\partial_{xx}$ is the Hessian with respect to $x$.
Under the typical assumptions of \cite[Chapter~5, Theorem~6.6]{yong1999stochastic},
    a feedback policy satisfying the inclusion
\begin{align}
    \pi^*(s,x) \in \argmin_{u \in U} \{ \ell(s,x,u) + f(s,x,u)^\top \partial_x V^*(s,x) \} \text{,}
    \label{eq:optpolicy}
\end{align}
    when $V^*$ is differentiable,\footnote{When $V^*$ is not differentiable a similar result
    can be obtained using the superdifferentials of the viscosity solution.} 
    is optimal, that is, $V^{\pi^*} \equiv V^*$.
More generally, when 
    $f^\mu, \ell^\mu$ are uniformly continuous and Lipschitz in $x$ for all ${t \in [0,T]}$
    (which we henceforth assume),
    the on-policy PDE \eqref{eq:hjpde} admits a unique viscosity solution 
    \cite[Chapter~7, Theorem~4.1]{yong1999stochastic}.
    \footnote{We assume without proof that the space of policies $\mu$ for which this
        condition is satisfied either contains $\pi^*$ or a close approximation of it.}

\subsection{On-Policy FBSDE}
The positivity of $\sigma\sigma^{\top}$ yields that \eqref{eq:hjpde} is a parabolic PDE and, 
    hence, by the Feynman-Kac Theorem (see, e.g. \cite{peng1991probabilistic}) its solution is 
    linked to the solution $(X_s,Y_s,Z_s)$ of the pair of FBSDEs composed of the 
    FSDE \eqref{eq:fsdeorig} and the backward SDE (BSDE)
\begin{align}
    \dY_s &= -\ell^\mu_s \, \ds + Z^{\top}_s \, \dW_s^{\Q}, & Y_T &= g(X_T)  \text{,} 
    \label{eq:bsdeorig}
\end{align}
    where $Y_s$ and $Z_s$ are, respectively, one and $n$-dimensional adapted processes.
\begin{theorem}[Feynman-Kac Representation] \label{thm:continfk}
    The solution 
    $(X_s,Y_s,Z_s)$ of the \mbox{FBSDE} system
    \eqref{eq:fsdeorig} and \eqref{eq:bsdeorig}
    satisfies
\begin{gather}
    \begin{aligned}
        Y_s &= V^\mu(s,X_s) \text{,} & s \in [0,T] \text{,} \\
        Z_s &= \sigma_s^\top \partial_x V^\mu(s,X_s) \text{,} & \text{a.e.} \;  s \in [0,T] \text{,}
    \end{aligned} \label{eq:yzv}
\end{gather}
    $\Q$-almost surely (a.s.).
\hfill $\square$
\end{theorem}
\begin{proof}
    See \cite[Chapter~7, Theorem~4.5, (4.29)]{yong1999stochastic}.
\end{proof}
In numerical methods this theorem is often applied over short time intervals,
    leading to the following result.
\begin{corollary}
\label{thm:DeltaYcorollary}
    Let $(X_s,Y_s,Z_s)$ be the solution to the \mbox{FBSDE} system
    \eqref{eq:fsdeorig} and \eqref{eq:bsdeorig} and define
\begin{align}
    \wh{Y}_{t,\tau} &:= Y_\tau - \Delta \wh{Y}_{t,\tau} \text{,} 
    \label{eq:bsdediff}
\end{align}
    where $\Delta \wh{Y}_{t,\tau}$ is either
\begin{align}
    \Delta \wh{Y}_{t,\tau}^\textnormal{noisy} 
    &:= - \int_t^\tau \ell^\mu_s \ds +  \int_t^\tau Z^\top_s \dW_s^\Q \text{,} \label{eq:bsdediffnoisy} \\
    \intertext{or} 
    \Delta \wh{Y}_{t,\tau}^\textnormal{noiseless} 
    &:= - \int_t^\tau \ell^\mu_s \ds \text{.}  
    \label{eq:bsdediffnoiseless}
\end{align}
    Then,
\begin{align}
    Y_t &= \E_{\Q}[\wh{Y}_{t,\tau}^\textnormal{noisy}| X_t]
    = \E_{\Q}[\wh{Y}_{t,\tau}^\textnormal{noiseless}| X_t] = V^\mu(t,X_t) \text{,}
    \label{eq:thmcontinfk}
\end{align}
    $\Q$-a.s. for $0 \leq t \leq \tau \leq T$.
\hfill $\square$
\end{corollary}

\begin{proof}
    The fact that $Y_t = \wh{Y}_{t,\tau}^\textnormal{noisy}$, $\Q$-a.s., follows directly from
    the definition of an SDE solution. 
    Since $Y_t$ (and consequently $\wh{Y}_{t,\tau}^\textnormal{noisy}$) 
        is $X_t$-measurable due to \eqref{eq:yzv}, it follows that
    $Y_t = \E_{\Q}[\wh{Y}_{t,\tau}^\textnormal{noisy}| X_t]$.
    The equality $\E_{\Q}[\wh{Y}_{t,\tau}^\textnormal{noisy}| X_t]
    = \E_{\Q}[\wh{Y}_{t,\tau}^\textnormal{noiseless}| X_t]$ follows immediately from the standard property of It\^o integrals (and the tower property of conditional expectation) that yields $\E_{\Q}[\int_t^\tau Z^\top_s \dW_s^\Q| X_t] = 0$ \cite[Chapter~7, Theorem~3.2]{yong1999stochastic}.
\end{proof}

\subsection{Least Squares Monte Carlo} \label{sec:lsmc}

Least squares Monte Carlo (LSMC) is a scheme for obtaining the parameters of a parametric model of the 
    value function $V^\mu$,
    originally credited to \cite{Longstaff2001}. 
\begin{corollary}\label{thm:lsmc}
The minimizer $\phi^*$ of 
\begin{align}
    &\inf_{\phi \in L_2} 
    \E_{\Q}[(\wh{Y}_{t,\tau} - \phi)^2] \text{,}
    \label{eq:thmcontinfk4}
\end{align}
    over $X_t$-measurable square integrable variables $\phi$ coincides with
    the value function, that is, $\phi^* = V^\mu(t,X_t)$.
\hfill $\square$
\end{corollary}
\begin{proof}
    This follows from the $L_2$-projective properties of conditional expectation \cite[Chapter~10.3, Property~11]{resnick2003probability} applied to~\eqref{eq:thmcontinfk}. 
\end{proof}

In LSMC numerical methods, we approximate the minimization in \eqref{eq:thmcontinfk4}
    over the subspace of $X_t$-measurable variables $\{\phi(X_t;\alpha) : \alpha \in \mathcal{A}\}$,
    where $\phi(x;\alpha)$ is a function representation with parameters 
$\alpha \in \mathcal{A}$ (we assume henceforth that \mbox{$\phi(x;\alpha) \in C^2(\R^n)$} for all
    $\alpha \in \mathcal{A}$).
Let $\{ (x^k_t,\wh{y}^k_{t}) \}_{k = 1}^M$
    be a set of samples approximating the joint distribution $(X_t, \wh{Y}_{t,\tau})$,
    denoted as $\wt{\Q}$.
The optimal parameters for this representation are found by minimizing
\begin{align}
    \alpha_t^* 
    &:= \argmin_{\alpha \in \mathcal{A}} 
    \E_{\wt{\Q}}[(\wh{Y}_{t,\tau} - \phi(X_t; \alpha_t))^2] \nonumber \\
    & \approx \argmin_{\alpha \in \mathcal{A}} \sum_{k = 1}^M
    \frac{1}{M} (\wh{y}^k_t - \phi(x^k_t; \alpha_t))^2 \label{eq:alphastar} 
    \text{.}
\end{align}
When the function representation is linear in the parameters $\phi(x;\alpha) = \Phi(x) \alpha$
    this optimization is a linear least squares regression problem in $\mathcal{A}$.
The optimal parameters define the new approximate representation of the value function, by
\begin{align}
    V^\mu(t,x) \approx \wt{V}^\mu(t,x) := \phi(x;\alpha^*_t) \text{.}
\end{align}

\subsection{Off-Policy Drifted FBSDE} \label{sec:driftfbsde}

We now present a result based on Girsanov's theorem, namely, that an alternative pair of 
    \textit{drifted} FBSDEs with a different trajectory
    distribution can be used to estimate the 
    same value function $V^\mu$. 
This result will be used to disentangle the drift of the forward distribution from
    the policy associated with the value function.
\begin{theorem} \label{thm:driftfbsde}
    Let ${(\Omega, \F, \{\F_t\}_{t \in [0,T]}, \P)}$ be a new filtered probability space on
    which $W^\P_s$ is Brownian and
    let $K_s$ be any $\F_s$-progressively measurable process on the interval $[0,T]$
    such that
\begin{align}
D_s &:= \sigma_s^{-1} (f^\mu_s - K_s) \text{,}
\end{align}
    is bounded and
\begin{align}
    \dX_s &= K_s \, \ds + \sigma_s \, \dW^\P_s, & X_0 &= x_0 \text{,}
    \label{eq:driftfsde}
\end{align}
    admits a unique square-integrable solution 
    $X_s$ (see e.g. \cite[Chapter~1, Theorem~6.16]{yong1999stochastic}).
Then, the Hamilton-Jacobi PDE \eqref{eq:hjpde} has a representation as the unique square-integrable solution $(X_s,Y_s,Z_s)$ to the \mbox{FBSDEs} 
\eqref{eq:driftfsde} and
\begin{align}
    \dY_s &= -(\ell^\mu_s + Z^{\top}_s D_s) \, \ds + Z^{\top}_s \dW_s^{\P}, & Y_T &= g(X_T)  \text{,}     \label{eq:driftbsde}
\end{align}
in the sense that
\begin{gather}
    \begin{aligned}
        Y_s &= V^\mu(s,X_s) \text{,} & s \in [0,T] \text{,} \\
        Z_s &= \sigma_s^\top \partial_x V^\mu(s,X_s) \text{,} & \text{a.e.} \;  s \in [0,T] \text{,}
    \end{aligned} \label{eq:yzv2}
\end{gather}
    $\P$-a.s..
\hfill $\square$
\end{theorem}
\begin{proof}
The existence of a square-integrable solution to \eqref{eq:driftfsde} 
    allows the conditions of \cite[Chapter~7, Theorem~3.2]{yong1999stochastic}
    to be satisfied for \eqref{eq:driftbsde}, guaranteeing a unique square-integrable solution
    $(Y_s,Z_s)$.
Now define the processes
\begin{align}
    W_t^\Q &:= W_t^\P - \int_0^t D_s \, \mathrm{d} s, \label{eq:Wdef} \\
    \Theta_t &:= \exp \bigg ( -\frac{1}{2} \int_0^t \| D_s \|^2 \, \mathrm{d} s
    + \int_0^t D_s^\top \mathrm{d}W^\P_s  \bigg ), \label{eq:thetadef}
\end{align}
    for $t \in [0,T]$.
Since $D_s$ is bounded, 
    Girsanov's theorem \cite[Chapter 5, Theorem 10.1]{fleming1976deterministic} 
    implies that the process $W^\Q_s$ defined by \eqref{eq:Wdef} is Brownian in 
    some measure $\Q$ derived from $\P$ in the form of
\begin{align}
    \mathrm{d} \Q &= \Theta_T \, \mathrm{d} \P \text{,} \label{eq:measuredef}
\end{align}
    where $\Theta_T$ be the Radon-Nikodym derivative.
With a simple algebraic reduction, Girsanov's theorem also guarantees separately that
    $X_s$ solves the on-policy FSDE \eqref{eq:fsdeorig}, and that
    $(X_s,Y_s,Z_s)$ solves the on-policy BSDE \eqref{eq:bsdeorig}.
Here, the idea is that the sample functions for the processes are the same, but
    the probability measure (acting on sets of trajectory samples~$\omega$)
    which characterizes their distributions changes.

Since $D_s$ is bounded, it satisfies Novikov's criterion \cite[Theorem~15.4.2]{CohenElliott} and thus, 
    it follows that
    $\Theta_t$ is $\P$-a.s. strictly positive, and further that
    the measures $\P$ and $\Q$
    are equivalent, that is, they are absolutely continuous with respect to the other 
    \cite{lowther_2011}. 
Since \eqref{eq:yzv} holds $\Q$-a.s., 
    there exists an $N \in \F$ such that $E^\mathsf{c} \subseteq N$,
    where $E := \{ \omega \in \Omega : Y_t(\omega) = V^\mu(t,X_t(\omega)) \}$,
    and $\Q(N) = 0$.
It subsequently follows from the definition of absolute continuity that $\P(N) = 0$.
Thus, \eqref{eq:yzv} holds $\P$-a.s. as well. 
\end{proof}

As before, the corresponding relationship over short intervals follows.

\begin{corollary}
    Let $(X_s,Y_s,Z_s)$ be the solution to the drifted \mbox{FBSDE} system
    \eqref{eq:driftfsde} and \eqref{eq:driftbsde}, and define
\begin{align}
    \wh{Y}_{t,\tau} &:= Y_\tau - \Delta \wh{Y}_{t,\tau} \text{,} 
    \label{eq:bsdediffdrifted}
\end{align}
    where $\Delta \wh{Y}_{t,\tau}$ is either
\begin{align}
    \Delta \wh{Y}_{t,\tau}^\textnormal{noisy} 
    &:= - \int_t^\tau (\ell^\mu_s + Z_s^\top D_s) \ds +  \int_t^\tau Z^\top_s \dW_s^\P \text{,} 
    \label{eq:bsdediffnoisydrifted} \\
    \intertext{or} 
    \Delta \wh{Y}_{t,\tau}^\textnormal{noiseless} 
    &:= - \int_t^\tau (\ell^\mu_s + Z_s^\top D_s) \ds \text{.}  
    \label{eq:bsdediffnoiselessdrifted}
\end{align}
    Then,
\begin{align}
    Y_t &= \E_{\P}[\wh{Y}_{t,\tau}^\textnormal{noisy}| X_t]
    = \E_{\P}[\wh{Y}_{t,\tau}^\textnormal{noiseless}| X_t] = V^\mu(t,X_t) \text{,}
    \label{eq:thmcontinfkdrifted}
\end{align}
    $\P$-a.s. for $0 \leq t \leq \tau \leq T$.
\hfill $\square$
\end{corollary}

\begin{proof}
The proof follows similarly to the proof of Corollary~\ref{thm:DeltaYcorollary}. 
\end{proof}

Further, the discussion in Section~\ref{sec:lsmc} holds true when the measure $\Q$ is replaced with $\P$.
We can interpret this result in the following sense.
As long as the diffusion function $\sigma$ is the same as in the on-policy formulation, we can pick an arbitrary process $K_s$ to be the drift term, which generates a distribution for the forward process $X_s$ in the corresponding measure $\P$. The BSDE yields an expression for $Y_t$ using the same process $W^{\P}_s$ as used in the FSDE. The term $Z^{\top}_s D_s$ acts as a correction in the BSDE to compensate for changing the drift of the FSDE. We can again use the minimization \eqref{eq:alphastar} to approximate the value function $V^\mu$, the only difference being that $(x_t^k,\wh{y}^k_t)$ are now samples approximating the distribution $\P$. 

It should be highlighted that $K_s$ need not be a deterministic function of the random variable $X_s$, as is the case with $f^\mu_s$. For instance, it can be selected as the function \mbox{$K_s(\omega) = h(s,X_s(\omega),\omega)$} for some appropriate function $h$, producing a non-trivial joint distribution for the random variables $(X_t,K_t)$.

\subsection{Euler-Maruyama FBSDE Approximation}

Many approaches to solving the FBSDEs
    propose approximating
    both the forward and backward steps with Euler-Maruyama-like SDE approximations,
    see, for instance, \cite{Bender2010}, \cite{exarchos2018stochastic}, and the survey in \cite{Higham2015}.
For the drifted FSDE the approximation is
\begin{align}
        X_\tau - X_t &= K_t \, \Delta t + \sqrt{\Delta t} \, \sigma_t \, \Delta W_t \text{,} 
\end{align}
where $\Delta t := \tau - t$ and $\Delta W_t \sim \mathcal{N}(0, I_n)$. 
For the drifted BSDE step we have
\begin{align}
    \wh{Y}_{t,\tau} 
    &= V(\tau,X_\tau) - \Delta \wh{Y}_{t,\tau} \text{,} 
    \label{eq:bsdediffdrifted2}
\end{align}
    where $\Delta \wh{Y}_{t,\tau}$ is either
\begin{align}
    \Delta \wh{Y}_{t,\tau}^\textnormal{noisy} 
    &= -(\ell^\mu_t + Z_{\tau}^\top D_t) \, \Delta t + Z^\top_{\tau} \sqrt{\Delta t} \, \Delta W_t  \text{,} 
    \label{eq:eulermarunoisy} \\
    \intertext{or} 
    \Delta \wh{Y}_{t,\tau}^\textnormal{noiseless} 
    &= -(\ell^\mu_t + Z_{\tau}^\top D_t) \, \Delta t  \text{.} \label{eq:eulermarunoiseless}
\end{align}
The variable $Z_\tau$ is evaluated at the end of the interval so that it can utilize the latest
    approximation of the value function gradient.
Note also that the on-policy Euler-Maruyama estimators arise when
    $K_t \equiv F^\mu_t$ and thus $D_t \equiv 0$.
The primary contribution of this paper, discussed in the next section, 
    is to propose new estimators for $\wh{Y}_{t,\tau}$
    to be used in the LSMC function regression step.

\section{Forward-Backward Difference Equations} \label{sec:fbdeqs}

In the previous section we presented results from continuous-time FBSDE theory, then used standard
    methods in SDE approximation to form a discrete-time approximation of the forward and backward
    SDEs.
In this section we propose the converse approach: we begin by forming a discrete-time approximation
    of the dynamics and the value function, then we derive relationships which resemble those
    arrived at previously.
In doing so, we make two contributions: first, we arrive at better estimators 
    compared to the direct discretization of the continuous time relations because we are
    able to exploit characteristics of the discrete-time formulation obscured by the continuous-time problem,
    and, secondly, we provide a discrete-time intuition for the continuous-time
    theory.

\subsection{Discrete Time SOC Approximation} \label{sec:dtsocapprox}

The interval $[0,T]$ is partitioned into $N$ subintervals of length $\Delta t$ with the partition
    $\{ t_0 = 0, t_1 = \Delta t, ..., t_{N-1} = T - \Delta t, t_N = T \}$.
We abbreviate variables $X_{t_i} =: X_i$ for brevity.
Let ${(\wt{\Omega}, \wt{\F}, \{\wt{\F}_i\}_{i \in \{0,\ldots,N\}}, \wt{\Q})}$ 
    be the discrete-time filtered probability space
    and let $\{W_i^\Q\}_{i=0}^{N-1}$ be a discrete time Brownian process in $\wt{\Q}$, that is,
    $W_i^\Q \sim \mathcal{N}(0,I_n)$ is normally distributed, $\wt{\F}_{i+1}$-measurable,
    and $\{W_i^\Q\}$ are mutually independent.
The on-policy forward stochastic difference equation is 
\begin{align}
    X_{i+1} - X_i &= F^\mu_i + \Sigma_i W^\Q_i \text{,} & X_0 &= x_0 \text{,}
    \label{eq:onpolicyfsdedt}
\end{align}
where, using the Euler-Maruyama approximation method,\footnote{Or some other 
    approximation scheme that results in the form \eqref{eq:onpolicyfsdedt}, \eqref{eq:onpolicyvalfn}.} 
\begin{align}
    F^\mu_i &= f(t_i,X_i,\mu_i(X_i)) \Delta t, &
    \Sigma_i &= \sigma(t_i,X_i) \sqrt{\Delta t} \text{,}
\end{align}
    and the on-policy value function is
\begin{align}
    V^\mu_i(X_i) &= \E_{\wt{\Q}}[\, \sum^{N-1}_{j=i} L^\mu_j  + g(X_N) \, | X_i] \text{,}
    \label{eq:onpolicyvalfn}
\end{align}
    where
\begin{align}
    L^\mu_j &= \ell(t_j,X_j,\mu_j(X_j)) \Delta t \text{.}
\end{align}
%

According to \cite[Chapter 10, Theorem 10.2.2]{kloeden2013numerical},
    when a linear growth  condition in $x$ is imposed on $f^\mu_s$, $\sigma_s$, and $\ell^\mu_s$
    along with a few other conditions, then it can be shown 
    that the absolute error between the Euler-Maruyama approximation $X_i$ and 
    the continuous forward process $X_t$ is 
    of order
    $\mathcal{O}((\Delta t)^{\sfrac{1}{2}})$.
When $\sigma_s$ is constant with respect to $x$, 
    the error bound improves to 
    $\mathcal{O}(\Delta t)$
    \cite[Chapter 10, Theorem 10.3.5]{kloeden2013numerical}.

\subsection{Discrete-Time BSDE Approximation}

For the discrete-time value function $\{V_i^\mu\}$ and forward process $\{X_i\}$
    we define the process $\{Y_i := V^\mu_i(X_i) \}$.
Further, we define the term $\Delta Y_i$ as one that satisfies the backward difference,
\begin{align}
    \Delta Y_i &:= Y_{i+1} - Y_i,
    \label{eq:truebackstep}
\end{align}
where we use separate estimators ${\wh{Y}_{i+1} \approx Y_{i+1}}$ and 
    $\Delta \wh{Y}_i \approx \Delta Y_i$ to obtain a combined estimator
\begin{align}
    \wh{Y}_i &:= \wh{Y}_{i+1} - \Delta \wh{Y}_i \text{.} \label{eq:backstep}
\end{align}
    with the interpretation $\wh{Y}_{i} \approx V^\mu_{i}(X_{i})$.
Both $\wh{Y}_{i+1}$ and $\Delta \wh{Y}_i$ can be chosen according to different approximation
    schemes; these choices are investigated below.
These approximation schemes assume the availability of some approximate representation of the 
    value function at the next step
    $\wt{V}^\mu_{i+1} \approx V^\mu_{i+1}$,
    as well as its derivatives, and they produce a representation 
    $\wt{V}^\mu_{i} \approx V^\mu_{i}$ using LSMC.

\subsection{On-Policy Taylor-Expanded Backward Difference}\label{sec:taylorBSDE}

We now propose an estimator for $\Delta \wh{Y}_i$, the discrete analogue to the
    on-policy terms defined in \eqref{eq:bsdediffnoisy} and \eqref{eq:bsdediffnoiseless}.
We begin by noting that the on-policy value function satisfies the on-policy Bellman equation
\begin{align}
    V^\mu_i(X_i) &= 
    L^\mu_i + \E_{\wt{\Q}}[ V^\mu_{i+1}(X_{i+1}) | X_i] \label{eq:opbellman}
    \text{.}  
\end{align}
Consider the second-order Taylor expansion of the approximation $\wt{V}^\mu_{i+1} \approx V^\mu_{i+1}$ 
    of the term inside the 
    conditional expectation,
\begin{align}
    \wt{V}^\mu_{i+1}(X_{i+1}) 
    &= \wt{V}^\mu_{i+1}(\ol{X}^\Q_{i+1} + \Sigma_i W_i^\Q) 
    = \wt{Y}_{i+1}
    + \delta^{\rm h.o.t.}_{i+1} \text{,} \label{eq:errwtY} \\
    \wt{Y}_{i+1} &:= \ol{Y}_{i+1}
        + \ol{Z}_{i+1}^\top W^\Q_i 
        + \frac{1}{2} W^{\Q \top}_i \ol{M}_{i+1} W^\Q_i \label{eq:wtY} 
    \text{,} 
\end{align}
    centered at the conditional mean, 
\begin{align}
    \ol{X}^\Q_{i+1} &:= \E_{\wt{\Q}}[X_{i+1} | X_i] =  X_i + F^\mu_i \text{,}
\end{align}
    where
\begin{align}
    \ol{Y}_{i+1} &:= \wt{V}^\mu_{i+1}(\ol{X}^\Q_{i+1}) \text{,} \label{eq:defY} \\
    \ol{Z}_{i+1} &:= \Sigma_i^\top \partial_x \wt{V}^\mu_{i+1}(\ol{X}^\Q_{i+1}) \text{,} \label{eq:defZ} \\
    \ol{M}_{i+1} &:= \Sigma_i^\top \partial_{xx} \wt{V}^\mu_{i+1}(\ol{X}^\Q_{i+1}) \Sigma_i 
    \label{eq:defM} \text{,}
\end{align}
    and $\delta^{\rm h.o.t.}_{i+1}$ includes the third and higher order terms in the Taylor
    series expansion.
Substituting $\wt{Y}_{i+1}$ in for $V^\mu_{i+1}(X_{i+1})$ in \eqref{eq:opbellman} and rearranging
    terms, and in light of \eqref{eq:backstep}, we arrive at an estimator for the backward step
\begin{align}
    \Delta \wh{Y}_{i}^{\textnormal{taylor}} &:= 
    - L^\mu_i + \ol{Z}_{i+1}^\top W^\Q_i \nonumber \\
    &\quad + \frac{1}{2} \tr(\ol{M}_{i+1} (W^\Q_i W^{\Q \top}_i - I)) \text{.}
    \label{eq:dtbsde}
\end{align}
For the purposes of comparison we restate the on-policy Euler-Maruyama estimators derived 
    in the previous section,
\begin{align}
    \Delta \wh{Y}_{i}^{\textnormal{noisy-em}} &:= 
    - L^\mu_i + \wt{Z}_{i+1}^\top W^\Q_i  \text{,} \\
    \Delta \wh{Y}_{i}^{\textnormal{nless-em}} &:= 
    - L^\mu_i  \text{,} 
\end{align}
    where
\begin{align}
    \wt{Z}_{i+1} &:= \Sigma_i^\top \partial_x \wt{V}^\mu_{i+1}(X_{i+1}) \text{.} 
\end{align}
There are two differences in the proposed Taylor series expansion approach compared to the 
    Euler-Maruyama approach.
First, the gradient of the value function is evaluated at $\ol{X}^\Q_{i+1}$ instead of $X_{i+1}$.
This effect can be exploited because in the discrete-time approach the difference equation separates the
    drift step and the diffusion step, whereas in the continuous-time approach the drift and
    diffusion are considered inseparable.
However, if the continuous-time SDEs are eventually discretized using Euler-Maruyama, this
    assumption is broken over short intervals.
Secondly, the trace term now appears in the Taylor-expansion estimator.
While in the continuous-time counterpart second-order effects are infinitesimally small, they can no longer be ignored in the discrete-time approximation.
Note, however, that 
$\E_{\Q}\big[\frac{1}{2} \tr(\ol{M}_{i+1} (W^\Q_i W^{\Q \top}_i - I))\big| X_i \big] = 0$
 since $\E_{\Q}[W^\Q_i W^{\Q \top}_i | X_i] = I$ and $\ol{M}_{i+1}$ is $X_i$-measurable.

The following theorem suggests that this choice of approximation of $\Delta Y_i$ has relatively
    small residual error.
    
\begin{theorem} \label{thm:deltaest}
    The choice $\Delta \wh{Y}_{i}^{\textnormal{taylor}}$ in \eqref{eq:dtbsde} is an unbiased estimator of the actual 
    value function difference $\Delta Y_i$, i.e.,
\begin{align}
    \E_{\wt{\Q}}[\Delta \wh{Y}_i|X_i] 
    &= \E_{\wt{\Q}}[\Delta Y_i|X_i] \text{.} \label{eq:unbias1}
\end{align}
Further, the residual error is
\begin{align}
    \Delta Y_i - \Delta \wh{Y}_i &= 
    \delta^{\Delta \wh{Y}}_{i+1}
    - \E_{\wt{\Q}}[\delta^{\Delta \wh{Y}}_{i+1} |X_i] \text{,} \label{eq:deltadelta1} \\
    \delta^{\Delta \wh{Y}}_{i+1} &:= \delta^{\wt{V}}_{i+1} + \delta^\textnormal{h.o.t.}_{i+1}
    \text{,} \label{eq:deltawhydef}
\end{align}
    where $\delta^{\wt{V}}_{i+1} := V^\mu_{i+1}(X_{i+1}) - \wt{V}^\mu_{i+1}(X_{i+1})$ 
    is the error in the $(i+1)^{\text{st}}$ step value function representation.
\hfill $\square$
\end{theorem}

\begin{proof}
The relationship \eqref{eq:unbias1} follows directly from taking the conditional expectation
    $\E_{\wt{\Q}}[\, \cdot \, |X_i]$ of both sides of \eqref{eq:deltadelta1}.
We now show \eqref{eq:deltadelta1}.

Comparing \eqref{eq:dtbsde} with \eqref{eq:wtY}, it can be easily shown that
\begin{align}
    \Delta \wh{Y}_i = - L^\mu_i + \wt{Y}_{i+1} - \E_{\wt{\Q}}[\wt{Y}_{i+1}|X_i] \text{,}
    \label{eq:deltayi}
\end{align}
    and, similarly, the Taylor expansion \eqref{eq:errwtY} immediately yields
    $Y_{i+1} = \wt{Y}_{i+1} + \delta^{\Delta \wh{Y}}_{i+1}$.
Combining these two expressions yields
\begin{align}
    \Delta \wh{Y}_i &= - L^\mu_i + Y_{i+1} - \delta^{\Delta \wh{Y}}_{i+1} 
    - \E_{\wt{\Q}}[Y_{i+1} - \delta^{\Delta \wh{Y}}_{i+1}|X_i] \text{.}
\end{align}
Substituting in the Bellman equation \eqref{eq:opbellman} and rearranging we
    arrive at \eqref{eq:deltadelta1}.
\end{proof}

In general, the Taylor expansion residual $\delta^\textnormal{h.o.t.}_{i+1}$ has a small mean due
    to the following result.
\begin{proposition} \label{cor:odd}
    Of the higher order terms in the Taylor expansion residual $\delta^\textnormal{h.o.t.}_{i+1}$,
        the terms with odd order, 
        starting with the third order term, have zero 
        conditional expectations given $X_i$.
\hfill $\square$
\end{proposition}
\begin{proof}
    See Appendix \ref{sec:cor_odd}.
\end{proof}

Further, under a very basic function approximation scheme, we can entirely dismiss the term $\delta^\textnormal{h.o.t.}_{i+1}$.

\begin{proposition} \label{cor:quadzero}
    If the value function approximation $\wt{V}^\mu_{i+1}$ is quadratic then
    $\delta^\textnormal{h.o.t.}_{i+1} \equiv 0$.
    Thus, the residual error is determined entirely by the residual error of the
    function approximation of $V^\mu_{i+1}$,
\begin{align}
    \Delta Y_i - \Delta \wh{Y}^{\textnormal{taylor}}_i = 
    \delta^{\wt{V}}_{i+1}
    - \E_{\wt{\Q}}[\delta^{\wt{V}}_{i+1} |X_i] \label{eq:resid3} 
    \text{.}
\end{align}
\hfill $\square$
\end{proposition}
\begin{proof}
    This is a direct consequence of the fact that if $\wt{V}^\mu_{i+1}$ is quadratic
    then its second order Taylor expansion is exact.
\end{proof}

Note that this does not require the true value function to be quadratic, only its approximation.
Although using a less expressive representation improves the error coming from
    the term $\delta^\textnormal{h.o.t.}_{i+1}$, there may be a trade-off in terms of increasing the magnitude of the error
    in $\delta^{\wt{V}}_{i+1}$, since the function $V^\mu_{i+1}$ might be less appropriately modeled.

The most remarkable aspect of Proposition~\ref{cor:quadzero} is that it suggests that for
    linear-quadratic-regulator (LQR) problems these estimators are exact up to function
    approximation error, due to the fact that for LQR problems $V^\mu_i$ itself
    is in the class of quadratic functions.
This provides a fundamental guarantee for these estimators.
On the contrary, the Euler-Maruyama estimators are not exact when applied to LQR problems.

\begin{remark} \label{cor:quadzero2}
    If the value function approximation $\wt{V}^\mu_{i+1}$ is quadratic,
    the residual error of the Euler-Maruyama estimators is 
\begin{align}
    &\Delta Y_i - \Delta \wh{Y}^{\textnormal{noisy-em}}_i \nonumber \\
    &\quad = 
    \delta^{\wt{V}}_{i+1}
    - \E_{\wt{\Q}}[\delta^{\wt{V}}_{i+1} |X_i] 
    + (\ol{Z}_{i+1} - \wt{Z}_{i+1})^\top W^\Q_i \nonumber\\
    &\quad \quad + \frac{1}{2} \tr(\ol{M}_{i+1} (W^\Q_i W^{\Q \top}_i - I))
    \label{eq:resid4} 
    \text{.}
\end{align}
\hfill $\square$
\end{remark}
Though all three $\Delta Y_i$ estimators are unbiased, the Taylor-expansion
    estimator is theoretically far superior on the baseline LQR problem.
In numerical experiments illustrated later we confirm this near-machine precision performance
    of the Taylor estimator and the divergence of the EM estimators on the same LQR problem.


\subsection{Estimators of $\wh{Y}_{i+1}$ }\label{sec:options}

We propose two potential estimators for $\wh{Y}_{i+1} \approx V^\mu_{i+1}(X_{i+1})$.
First, we propose using the value function approximation associated with the previous backward
    step to re-estimate the $\wh{Y}_{i+1}$ values,
\begin{align}
    \wh{Y}_{i+1}^{\textnormal{re-est}} &:= \wt{V}^\mu_{i+1}(X_{i+1}) \text{.} \label{eq:reinit} 
\end{align}
Alternatively, we can also use the estimator 
\begin{align}
    \wh{Y}_{i+1}^{\textnormal{noiseless}} &:= \wt{Y}_{i+1} \text{,} \label{eq:noiseest} 
\end{align}
which ends up cancelling out the terms with $W^\Q_i$ in them, so that \eqref{eq:backstep}
    reduces to
\begin{align}
    \wh{Y}^{\textnormal{noiseless}}_i 
    &= L^\mu_i + \ol{Y}_{i+1} + \frac{1}{2} \tr(\ol{M}_{i+1}) \text{.} 
    \label{eq:noiseless}
\end{align}
The following theorem establishes the error analysis of the two Taylor-expansion-based
    estimators.
\begin{theorem} \label{thm:biasvarhat}
For the estimator $\wh{Y}_i := \wh{Y}_{i+1} - \Delta \wh{Y}_i$, where
$\Delta \wh{Y}_i$ is defined in \eqref{eq:dtbsde} and $\wh{Y}_{i+1}$
is defined in \eqref{eq:reinit} or \eqref{eq:noiseest}, the bias is
\begin{align}
    \E_{\wt{\Q}}[Y_i - \wh{Y}^\textnormal{re-est}_i |X_i]
    &= \E_{\wt{\Q}}[\delta^{\wt{V}}_{i+1} |X_i] 
    \text{,} \\
    \E_{\wt{\Q}}[Y_i - \wh{Y}^\textnormal{noiseless}_i |X_i] 
    & = \E_{\wt{\Q}}[\delta^{\wt{V}}_{i+1} + \delta^\textnormal{h.o.t.}_{i+1} |X_i]  
    \text{.} 
\end{align}
    The variances of these estimators are
\begin{align}
    \var_{\wt{\Q}}[\wh{Y}^\textnormal{re-est}_i |X_i]
    &= \var_{\wt{\Q}}[\delta^\textnormal{h.o.t.}_{i+1} |X_i]
    \text{,} \\
    \var_{\wt{\Q}}[\wh{Y}^\textnormal{noiseless}_i |X_i]
    &= 0
    \text{.} 
\end{align}
\hfill $\square$
\end{theorem}

\begin{proof}
    See Appendix \ref{sec:ProofThm32n33}.
\end{proof}

This theorem shows that the \textit{re-estimate} condition has less bias than the
    \textit{noiseless} condition, but it is a higher variance estimator.
We also observe that when $\delta^\textnormal{h.o.t.}_{i+1} = 0$ the bias
    and variance of these two estimators are identical.
However, since it is not immediately clear which condition is superior when this is not true, 
    we examine both
    methods and compare the results in Section~\ref{sec:expresults}.

\subsection{Drifted Taylor-Expanded Backward Difference}

We now offer a discrete-time approximation of the drifted off-policy FBSDEs.
Let ${(\wt{\Omega}, \wt{\F}, \{\wt{\F}_i\}_{i \in \{0,\ldots,N\}}, \wt{\P})}$
    be an alternative discrete-time filtered probability space where $W_i^\P$ is
    the associated Brownian process.
Define on this space the difference equation 
\begin{align}
    X_{i+1} - X_i &= K_i + \Sigma_i W^\P_i \text{,} & X_0 &= x_0 ,
    \label{eq:driftfsdedt}
\end{align}
    where the process $\{K_i\}_{i=0}^{N-1}$ is chosen at will,
    $\wt{\F}_{i+1}$-measurable,
    and independent of $W_i^\P$. 
For example, $K_i$ can be constructed using the function 
    $K_i(\omega) = \mathcal{K}_i(X_i(\omega),\xi_i(\omega))$, where 
    $\{\xi_i\}$ is some random process where $\xi_i$ is $\wt{\F}_{i+1}$-measurable
    and independent of $W_i^\P$ (but not necessarily of $W_{i-1}^\P$). 
Each $K_i$ must also be selected such that
\begin{align}
    D_i &:= \Sigma_i^{-1} (F^\mu_i - K_i) \text{,}
\end{align}
    is bounded.

Similar to the construction in Section~\ref{sec:driftfbsde},
    a discrete time version of Girsanov's theorem can be used to produce the measure
    $\wt{\Q}$ which satisfies the assumptions of Section~\ref{sec:dtsocapprox} 
    and show how the drifted forward difference 
    \eqref{eq:driftfsdedt}
    can be transformed to the on-policy forward difference \eqref{eq:onpolicyfsdedt}.
    To this end, 
define the sequence of measures
\begin{align}
    \mathrm{d} \wt{\Q}_{i+1} = \Theta_{i+1} \mathrm{d} \wt{\P}_{i+1} \label{eq:dtmeas}
    \text{,}
\end{align}
    for $i=0, \ldots, N-1$, where 
    $\Theta_{i+1}$ is the discrete time version of \eqref{eq:thetadef} defined as
\begin{align}
    \Theta_{i+1} 
    &:= \exp \bigg (\sum_{j=0}^i \big(-\frac{1}{2} \| D_j \|^2 + D_j^\top W_j^\P \big) \bigg) \text{,} 
\end{align}
    where
    $\Theta_0 := 1$.
Further, as in \eqref{eq:Wdef}, define the process
\begin{align}
    W^\Q_i := W^\P_i - D_i \label{eq:wqwqdt}
    \text{,}
\end{align}
    for $i = 0, \ldots, N-1$.
    
\begin{lemma}[Discrete-Time Girsanov] \label{lemma:dtgirsanov}
Let $\{D_i\}_{i=0}^{N-1}$ be a sequence of bounded and $\wt{\F}_{i+1}$-measurable random variables where each $D_i$ is independent of $W_i^\P$. Then,
    the process $W^\Q_{0:i} := \{W^\Q_j\}_{j=0}^{i}$ 
defined by \eqref{eq:wqwqdt}
is Brownian with respect to $\wt{\Q}_{i+1}$, that
    is, $W_i^\Q \sim \mathcal{N}(0,I_n)$ is normally distributed, $\wt{\F}_{i+1}$-measurable,
    and the set of variables $\{W_i^\Q\}$ are mutually independent.
\hfill $\square$
\end{lemma}
\begin{proof}
    See Appendix~\ref{sec:dtgirsanov}.
\end{proof}


Henceforth, we let $\wt{\P} := \wt{\P}_N$ for notational simplicity.
It is easy to see that the drifted forward difference \eqref{eq:driftfsdedt} and  the on-policy
    forward difference \eqref{eq:onpolicyfsdedt}
    are identical under the substitution \eqref{eq:wqwqdt}.
Thus, we conclude that the drifted process $\{X_i\}$ still satisfies 
    the on-policy Bellman equation \eqref{eq:opbellman} for the same on-policy value function $V^\mu$.

To derive the backward step, we perform a Taylor expansion
    centered at 
\begin{align}
    \ol{X}^\P_{i+1} &:= \E_{\wt{\P}}[X_{i+1} | X_i, K_i] =  X_i + K_i \text{,}
\end{align}
    instead of $\ol{X}^\Q_{i+1}$.
The expressions defining $\wt{Y}_{i+1}$, $\ol{Y}_{i+1}$, $\ol{Z}_{i+1}$, and $\ol{M}_{i+1}$
    \eqref{eq:errwtY},\eqref{eq:wtY},\eqref{eq:defY},\eqref{eq:defZ},\eqref{eq:defM}
    are all identical except for replacing $\ol{X}^\Q_{i+1},W^\Q_i$ with $\ol{X}^\P_{i+1},W^\P_i$.
Again, substituting $\wt{Y}_{i+1}$ in for $V^\mu_{i+1}(X_{i+1})$ in \eqref{eq:opbellman} and rearranging
    terms in light of \eqref{eq:backstep}, we arrive at an estimator for the backward step
\begin{align}
    \begin{aligned}
        \Delta \wh{Y}_{i}^{\textnormal{drift}} &:= 
        - L^\mu_i + \ol{Z}_{i+1}^\top W^\P_i - \ol{Z}_{i+1}^\top D_i \\
        &\quad + \frac{1}{2} \tr(\ol{M}_{i+1} (W^\P_i W^{\P \top}_i - I - D_i D_i^\top)) \text{.}
        \label{eq:driftdtbsde}
    \end{aligned}
\end{align}
Recognize that this is a generalization of \eqref{eq:dtbsde},
    by noting that when $K_i = F^\mu_i$ then $D_i = 0$ and the drifted forward 
    difference \eqref{eq:driftfsdedt} 
    and the backward step reduce to their on-policy form \eqref{eq:onpolicyfsdedt}, \eqref{eq:dtbsde}.

\begin{lemma} \label{thm:driftdeltaest}
    The choice \eqref{eq:driftdtbsde} yields the residual error 
\begin{align}
    \Delta Y_i - \Delta \wh{Y}_i &= 
    \delta^{\Delta \wh{Y}}_{i+1}
    - \E_{\wt{\Q}}[\delta^{\Delta \wh{Y}}_{i+1} |X_i, K_i] \text{.} \label{eq:deltadelta2}
\end{align}
\hfill $\square$
\end{lemma}
\begin{proof}
Substituting \eqref{eq:wqwqdt} into
\begin{align}
    \wt{Y}_{i+1} &:= \ol{Y}_{i+1} 
        + \ol{Z}_i^\top W^\P_i 
        + \frac{1}{2} W^{\P \top}_i \ol{M}_i W^\P_i
        \text{,} \label{eq:new_wty}
\end{align}
yields
\begin{align}
    \wt{Y}_{i+1} 
    &= \ol{Y}_{i+1} 
        + \ol{Z}_{i+1}^\top (W^\Q_i + D_i) \nonumber \\
        &\quad + \frac{1}{2} (W^\Q_i + D_i)^\top \ol{M}_{i+1} (W^\Q_i + D_i) \nonumber \\
    &= \ol{Y}_{i+1} 
        + \ol{Z}_{i+1}^\top W^\Q_i + \ol{Z}_{i+1}^\top D_i 
        + D_i^\top \ol{M}_{i+1} W^\Q_i \nonumber \\
        &\quad + \frac{1}{2} \tr \big( \ol{M}_{i+1} (W^\Q_i W^{\Q \top}_i + D_i D_i^\top) \big) 
    \text{.}  \nonumber
\end{align}
Note that $D_i$, $\ol{Y}_{i+1}$, $\ol{Z}_{i+1}$, and $\ol{M}_{i+1}$, are $(X_i,K_i)$-measurable.
Taking the conditional expectation in the on-policy measure $\wt{\Q}$ yields
\begin{align}
    \E_{\wt{\Q}}[\wt{Y}_{i+1}|X_i, K_i] 
    &= \ol{Y}_{i+1}
        + \ol{Z}_{i+1}^\top D_i \nonumber \\
        &\quad + \frac{1}{2} \tr \big( \ol{M}_{i+1} (I + D_i D_i^\top) \big) 
        \text{.} \label{eq:condexpdrifty}
\end{align}
Comparing \eqref{eq:driftdtbsde}, \eqref{eq:new_wty}, and \eqref{eq:condexpdrifty}, 
    it can be easily shown that
\begin{align}
    \Delta \wh{Y}_i = - L^\mu_i + \wt{Y}_{i+1} - \E_{\wt{\Q}}[\wt{Y}_{i+1}|X_i, K_i] \text{.}
\end{align}
The Taylor expansion \eqref{eq:errwtY} immediately yields
    $Y_{i+1} = \wt{Y}_{i+1} + \delta^{\Delta \wh{Y}}_{i+1}$.
Combining these two expressions yields
\begin{align}
    \Delta \wh{Y}_i &= - L^\mu_i + Y_{i+1} - \delta^{\Delta \wh{Y}}_{i+1} 
    - \E_{\wt{\Q}}[Y_{i+1} - \delta^{\Delta \wh{Y}}_{i+1}|X_i, K_i] \text{.}
\end{align}
Substituting in the Bellman equation $Y_i = L^\mu_i + \E_{\wt{\Q}}[Y_{i+1} | X_i]$ 
    \eqref{eq:opbellman}  and rearranging,
    we have
\begin{align}
    &\Delta Y_i - \Delta \wh{Y}_i \nonumber \\
    &\quad = \delta^{\Delta \wh{Y}}_{i+1} 
    + \E_{\wt{\Q}}[Y_{i+1} - \delta^{\Delta \wh{Y}}_{i+1}|X_i, K_i] 
     - \E_{\wt{\Q}}[Y_{i+1} | X_i] 
    \text{.}
\end{align}
Under the measure $\wt{\Q}$, $Y_{i+1}$ is independent of $K_i$ given $X_i$, 
    so we have
\begin{align*}
    \E_{\wt{\Q}}[Y_{i+1}|X_i, K_i] 
    &= \E_{\wt{\Q}}[Y_{i+1} | X_i] 
    \text{,}
\end{align*}
    and by subsituting into the previous equation we arrive at \eqref{eq:deltadelta2}.
\end{proof}
The distribution of the residual error $\Delta Y_i - \Delta \wh{Y}_i$ 
    depends on the measure we use to interpret it.
For numerical applications we sample from the measure $\wt{\P}$ instead of $\wt{\Q}$, and thus
    this estimator is no longer unbiased with respect to the sampled distribution.
The conditional expectation with respect to $\wt{\P}$ of the right hand side of \eqref{eq:deltadelta2} is
\begin{align}
    \eps^{\P | \Q}_{i+1}
    &:= \E_{\wt{\P}}[\delta^{\Delta \wh{Y}}_{i+1}
    - \E_{\wt{\Q}}[\delta^{\Delta \wh{Y}}_{i+1} |X_i, K_i] |X_i, K_i] \nonumber \\
    &= \E_{\wt{\P}}[\delta^{\Delta \wh{Y}}_{i+1} | X_i, K_i]
    - \E_{\wt{\Q}}[\delta^{\Delta \wh{Y}}_{i+1} |X_i, K_i] 
    \text{.} 
\end{align}

The two estimators for $\wh{Y}_{i+1}$, \eqref{eq:reinit} \eqref{eq:noiseest},
    presented in Section~\ref{sec:options},
    can be used without modification, given that in the noiseless condition, 
    $\wt{Y}_{i+1}$ is taken to be \eqref{eq:new_wty}.
The drifted \textit{noiseless} estimator now resolves to
\begin{align}
    \wh{Y}^{\textnormal{noiseless}}_i 
    &= L^\mu_i + \ol{Y}_{i+1} + \ol{Z}_{i+1}^\top D_i \nonumber \\
    &\quad + \frac{1}{2} \tr(\ol{M}_{i+1} (I + D_i D_i^\top)) \text{.} 
    \label{eq:driftnoiseless}
\end{align}
\begin{theorem} \label{thm:newests}
    For the estimator $\wh{Y}_i := \wh{Y}_{i+1} - \Delta \wh{Y}_i$ where
    $\Delta \wh{Y}_i$ is defined in \eqref{eq:driftdtbsde} and $\wh{Y}_{i+1}$
    is defined in \eqref{eq:reinit} or \eqref{eq:noiseest} the bias is
\begin{align}
    \E_{\wt{\P}}[Y_i - \wh{Y}^\textnormal{re-est}_i |X_i, K_i]
    &= \E_{\wt{\Q}}[\delta^{\Delta \wh{Y}}_{i+1} |X_i, K_i]  \nonumber \\
    &\quad - \E_{\wt{\P}}[\delta^{\textnormal{h.o.t.}}_{i+1} |X_i, K_i]
    \text{,} \\
    \E_{\wt{\P}}[Y_i - \wh{Y}^\textnormal{noiseless}_i |X_i, K_i] 
    &= \E_{\wt{\Q}}[\delta^{\Delta \wh{Y}}_{i+1} |X_i, K_i]
    \text{.} 
\end{align}
    The variances of the estimators are
\begin{align}
    \var_{\wt{\P}}[\wh{Y}^\textnormal{re-est}_i |X_i, K_i]
    &= \var_{\wt{\P}}[ \delta^{\textnormal{h.o.t.}}_{i+1} |X_i, K_i] \\
    \var_{\wt{\P}}[\wh{Y}^\textnormal{noiseless}_i |X_i, K_i]
    &= 0
    \text{.} 
\end{align}
\hfill $\square$
\end{theorem}
\begin{proof}
    See Appendix \ref{sec:ProofThm32n33}.
\end{proof}

Since $\wt{\Q}$ is not available during computation, 
    we characterize $\E_{\wt{\Q}}[\delta^{\Delta \wh{Y}}_{i+1} |X_i]$
    exclusively in the measure $\wt{\P}$ using the next result.
    
\begin{proposition} \label{thm:diffmeaserr}
    The bias term appearing in Theorem~\ref{thm:newests} is bounded as
\begin{align}
    &|\E_{\wt{\Q}}[\delta^{\Delta \wh{Y}}_{i+1} |X_i, K_i]| \nonumber \\
    &\quad \leq \exp(\frac{1}{2}\| D_i \|^2) \, \,
    \E_{\wt{\P}}[(\delta^{\Delta \wh{Y}}_{i+1})^2 |X_i, K_i]^{1/2} 
    \text{.} \label{eq:epsbound}
\end{align}
\hfill $\square$
\end{proposition}
\begin{proof}
    See Appendix~\ref{sec:diffmeaserr}.
\end{proof}
Although the error bound in Proposition~\ref{thm:diffmeaserr} suggests that the bias grows rapidly with the magnitude $\| D_i \|$,
when this magnitude is small ($\| D_i \| \leq 1$) the first term 
    in the product on the right hand side of \eqref{eq:epsbound}
    is bounded by $\sqrt{e} \approx 1.65$. 
To illustrate the effect of $\| D_i \|$ on the error bound, consider a one-dimensional problem
    where we select $K_i = F^\mu_i + a$ for some random variable $a$ with bounded magnitude
    $| a | \leq \Sigma_i$ a.s..
It subsequently follows that \mbox{$\exp(\|D_i\|^2) \leq \sqrt{e}$} a.s..
This suggests that, in general, the magnitude of the difference $F^\mu_i - K_i$ should be proportional
    to the diffusion $\Sigma_i$.
Further, it is still the case that if the value function approximation is quadratic then the 
    higher order terms $\delta^{\textnormal{h.o.t.}}_{i+1}$ drop out.

These analytical results justify the assumption that for appropriately chosen $K_i$,
    the choice of \eqref{eq:driftdtbsde} represents a low bias, low variance approximator
    for the backward difference step.
It also provides guidance for how to select $K_i$.

\section{Policy Improvement} \label{sec:policyimprove}

In this section we discuss how policies can be improved based on the value function parameters
    obtained from the backward passes.
First, we discuss a na\"ive continuous-time approximation approach
    arising from the Hamiltonian used in HJB equations.
Continuous-time analysis of the Hamiltonian suggests that the optimal control policy $\pi^*$ satisfies
    the inclusion~\eqref{eq:optpolicy}, so a na\"ive approach to improving the policy
    would be to use the Euler-Maruyama approximation of the dynamics and costs along
    with the gradient of the recent approximation of the value function to
    evaluate this policy optimization.
This Hamiltonian-based approach
\begin{align}
    \wt{\pi}^*_i(x) 
    &\in \argmin_{u \in U} \{ \ell(t_i,x,u) + f(t_i,x,u)^\top \partial_x \wt{V}_{i}(x) \} \nonumber \\
    &\equiv \argmin_{u \in U} \{ L_i(x,u) + F_i(x,u)^\top \partial_x \wt{V}_{i}(x) \} 
    \text{,} \label{eq:hamil_policy}
\end{align}
    is used for estimating the optimal policy in 
    \cite{exarchos2018stochastic,Exarchos2018}.

 According to the discussion in the previous section, we propose an alternative Taylor-based approach to \eqref{eq:hamil_policy} as follows. 
 We begin with
    a discrete approximation of the continuous problem and form the Q-value function at time $i$, given the value function $V^\mu_{i+1}$, 
\begin{align}
    Q^\mu_i(x,u) := L_i(x,u) + \E_{\wt{\Q}_i}[V^\mu_{i+1}(X_{i+1})|X_i=x], \label{eq:qvalfnexact}
\end{align}
    where $\wt{\Q}_i$ is the measure corresponding to the forward difference step
\begin{align}
    X_{i+1} - x &= F_i(x,u) + \Sigma_i W^\Q_i \text{.} 
\end{align}
The optimal Bellman equation indicates that the optimal policy satisfies
    $\pi^*_i(x) \in \argmin_{u \in U} Q^{\pi^*}_i(x,u)$ and the optimal
    value function satisfies $V^{\pi^*}_i(x) = \min_{u \in U} Q^{\pi^*}_i(x,u)$.
Notice that when $V^\pi_{i+1} \leq V^\mu_{i+1}$ and $Q^\pi_i(x,\pi_i(x)) \leq Q^\mu_i(x,\mu_i(x))$ then
    $V^\pi_i \leq V^\mu_i$, so $\pi$ will be an improved policy over $\mu$.
Letting
\begin{align}
    \ol{X}^{x,u}_{i+1} &:= \E_{\wt{\Q}_i}[X_{i+1} | X_i = x] =  x + F_i(x,u) \text{,}
\end{align}
    and performing the same Taylor expansion approach as in \eqref{eq:errwtY}, \eqref{eq:wtY},
    we arrive at the approximation $\wt{Q}^\mu_i \approx Q^\mu_i$ defined as
\begin{align}
    \wt{Q}^\mu_i(x,u) 
    &:= L_i(x,u) + \ol{Y}_{i+1}^{x,u} + \frac{1}{2} \tr(\ol{M}_{i+1}^{x,u}) \text{,} 
    \label{eq:qvalfn}
\end{align}
    where
\begin{align}
    \ol{Y}_{i+1}^{x,u} &:= \wt{V}^\mu_{i+1}(\ol{X}^{x,u}_{i+1}) \text{,} \nonumber \\
    \ol{M}_{i+1}^{x,u} 
    &:= \Sigma_i^\top \partial_{xx} \wt{V}^\mu_{i+1}(\ol{X}^{x,u}_{i+1}) \Sigma_i \text{.} \nonumber
\end{align}

\begin{proposition}
    The error when using \eqref{eq:qvalfn} to approximate the Q-value function is
\begin{align}
    Q^\mu_i(x,u) - \wt{Q}^\mu_i(x,u) &= \E_{\wt{\Q}_i}[\delta^{\Delta \wh{Y}}_{i+1}|X_i=x]
    \text{.} \label{eq:qvalerr}
\end{align}
\hfill $\square$
\end{proposition}
\begin{proof}
    Due to \eqref{eq:errwtY} we have $Y_{i+1} = \wt{Y}_{i+1} + \delta^{\Delta \wh{Y}}_{i+1}$.
    We arrive at \eqref{eq:qvalerr} by subtracting \eqref{eq:qvalfn} from \eqref{eq:qvalfnexact}, 
    and then substituting in $Y_{i+1}$ for $V^\mu_{i+1}(X_{i+1})$.
\end{proof}

In general, we seek a policy that minimizes this Q-value function,
\begin{align}
    \mu^*(x;\wt{V}^\mu_{i+1}) := \min_{u \in U} \wt{Q}^\mu_i(x,u) \label{eq:optimprove} \text{.}
\end{align}
When the function $L_i$ is quadratic in terms of $u$ and/or contains an $L_1$ regularization term
    like $\sum_{j=1}^n |u^j|$, $U$~is~an interval set, $F_i$ is affine in the control, and 
    $\wt{V}^\mu_{i+1}$ is quadratic, then the optimization \eqref{eq:optimprove}
    has an analytic solution.
Also, similarly to the previous section, when $\wt{V}^\mu_{i+1}$ is quadratic, as is the
    case in LQR problems, the Taylor expansion of the Q-value function is exact.
Thus, this optimization will yield the exact optimal control solution for the LQR problem.

\section{Numerical Results} \label{sec:expresults}

Next, we numerically evaluate and compare the proposed Taylor estimators to the
    na\"ive Euler-Maruyama estimators on two problems, a nonlinear 1-dimensional problem
    and an LQR 4-dimensional problem.
The estimators discussed in this work are summarized in Table~\ref{tbl:estimators}.

\begin{table}[h]
    \centering
    \caption{Expressions for the proposed noiseless and re-estimate estimators, as well as the
             competing Euler-Maruyama estimators \eqref{eq:eulermarunoisy}, and \eqref{eq:eulermarunoiseless}
             (used in \cite{exarchos2018stochastic}).}
    \begin{tabular}{|c c|}
        \hline
        Estimator & $\wh{Y}_{i} = $ \\
        \hline \hline
        Taylor &
        $L^\mu_i + \ol{Y}_{i+1} + \ol{Z}_{i+1}^\top D_i$ \\ 
        Noiseless &
        $+ \frac{1}{2} \tr(\ol{M}_{i+1} (I + D_i D_i^\top))$  \\
        \hline
        Taylor &
        $\wt{V}^\mu_{i+1}(X_{i+1}) + L^\mu_i - \ol{Z}_{i+1}^\top W^\P_i + \ol{Z}_{i+1}^\top D_i$ \\ 
        Re-estimate &
        $+ \frac{1}{2} \tr(\ol{M}_{i+1} (I + D_i D_i^\top - W^\P_i W^{\P \top}_i))$ \\
        \hline
        Euler-Maru.&
        $\wt{V}^\mu_{i+1}(X_{i+1}) + L^\mu_i + \wt{Z}_{i+1}^\top D_i$ \\ 
        Noiseless \cite{exarchos2018stochastic} & \\
        \hline
        Euler-Maru. &
        $\wt{V}^\mu_{i+1}(X_{i+1}) + L^\mu_i - \wt{Z}_{i+1}^\top W^\P_i + \wt{Z}_{i+1}^\top D_i$ \\
        Noisy & \\
        \hline
    \end{tabular}
    \label{tbl:estimators}
\end{table}

\subsection{Nonlinear 1D Problem}

Consider the scalar optimal control problem with the dynamics and cost
\begin{align*}
    &\dX_s = \big(0.1 (X_s - 3)^2 + 0.2 u_s\big) \ds + 0.8\, \dW_s \text{,} \hspace{20pt} x_0 = 7 \text{,}
    \\
    &S_t(u_{[t,T]}) = \int^T_t \big(12 \, | X_s - 6 | + 0.4 \, u_s^2 \big) \, \ds + 25 \, X_T^2 \text{,} 
\end{align*}
over a time interval of length $T = 10$, with $N = 200$ discrete timesteps. 
We compute a ground-truth optimal value function $V^*_i$ by directly evaluating the optimal
    Bellman equation using a finely-gridded state space, control space, and noise space,
    and set the optimal policy as the target $\mu \equiv \pi^*$.
The optimal value function is visualized in Fig.~\ref{fig:nonlin_trajs} (the yellow surface), along
    with two forward-backward trajectory distributions $\{(X_i,Y_i)\}$ considered for evaluation:
    (a) the optimal 
    \mbox{$K_i^\text{optimal} = F^{\pi^*}_i$} (the cyan trajectories 
    in Fig.~\ref{fig:nonlin_trajs}), and
    (b) the suboptimal \mbox{$K_i^\text{subopt} = -0.2 X_i$} 
    (the orange trajectories).
We ran a series of simulations to investigate how each estimator performs under different
    algorithmic conditions, visualized in Fig.~\ref{fig:heatmaps}.
Each trial performs one forward pass, and then uses Chebyshev polynomials to 
    locally approximate the optimal value function in a single backward pass.
    For the purposes of evaluation we use the relative absolute error (RAE) metric \cite[Chapter~5]{witten2011}
\begin{align}
    &\frac{\int_{\mathcal{C}_i}  
        |\wt{V}_i(x) - V^*_i(x)| \, \mathrm{d} x}
    {\int_{\mathcal{C}_i}  
        |\int_{\mathcal{C}_i} V^*_i(y) \, \mathrm{d} y - V^*_i(x)| \, \mathrm{d} x} \nonumber \\
    &\quad \approx  
    \frac{\sum_{x \in \wt{\mathcal{C}}_i} |\wt{V}_i(x) - V^*_i(x)| }
    {\sum_{x \in \wt{\mathcal{C}}_i} |\sum_{y \in \wt{\mathcal{C}}_i} 
    \frac{1}{|\wt{\mathcal{C}}_i|} V^*_i(y) - V^*_i(x)|}
    \text{,} \label{eq:mad_def}
\end{align}
    where 
\begin{align}
    \mathcal{C}_i &:= [\underline{c}_i, \ol{c}_i] 
    := [ \ol{x}_i - \max\{3 \sigma_i , 1\},  \ol{x}_i + \max\{3 \sigma_i , 1\}] \nonumber \\
    &\approx \{ \underline{c}_i, \underline{c}_i + \Delta x, \ldots, 
        \ol{c}_i \} =: \wt{\mathcal{C}}_i
    \text{,} \label{eq:conf_reg} 
\end{align}
    for $i = 0, \ldots, N$, where $\ol{x}_i,\sigma_i$ are the mean and standard deviation
    of $X_i$ for the optimal forward trajectory distribution
    (the cyan trajectories in Fig.~\ref{fig:nonlin_trajs}).
For each element in Fig.~\ref{fig:heatmaps} we 
    average the RAE approximations \eqref{eq:mad_def} over both
    $20$ trials and $N=200$ timesteps.

\begin{figure}
    \centering
    \includegraphics[width=0.97\linewidth]{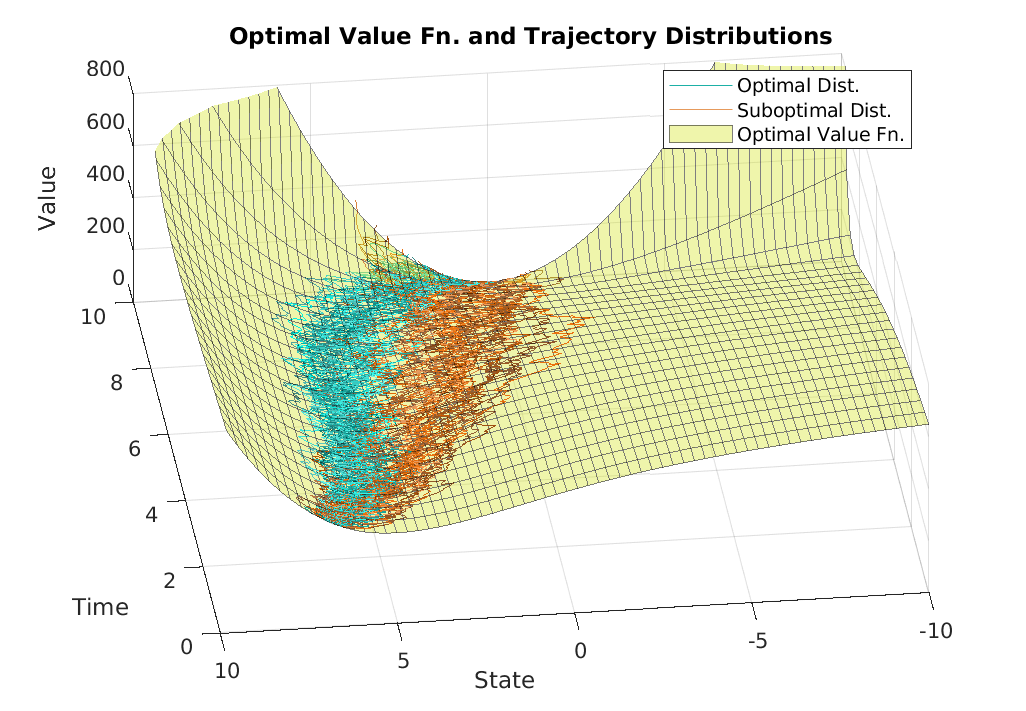}
    \caption{
        Optimal value function and trajectory distributions for the 1-dimensional nonlinear problem.
        The yellow surface is the ground truth optimal value function and the
        cyan and orange trajectories are the optimal and suboptimal trajectory distributions,
        respectively, used as forward distributions for evaluation.
    }
    \label{fig:nonlin_trajs}
\end{figure}

\begin{figure*}
    \centering
    \subfloat[Optimal forward sampling distribution generated with $K^\text{optimal}$ 
        (On-policy estimators).]
        {\includegraphics[width=0.47\linewidth]{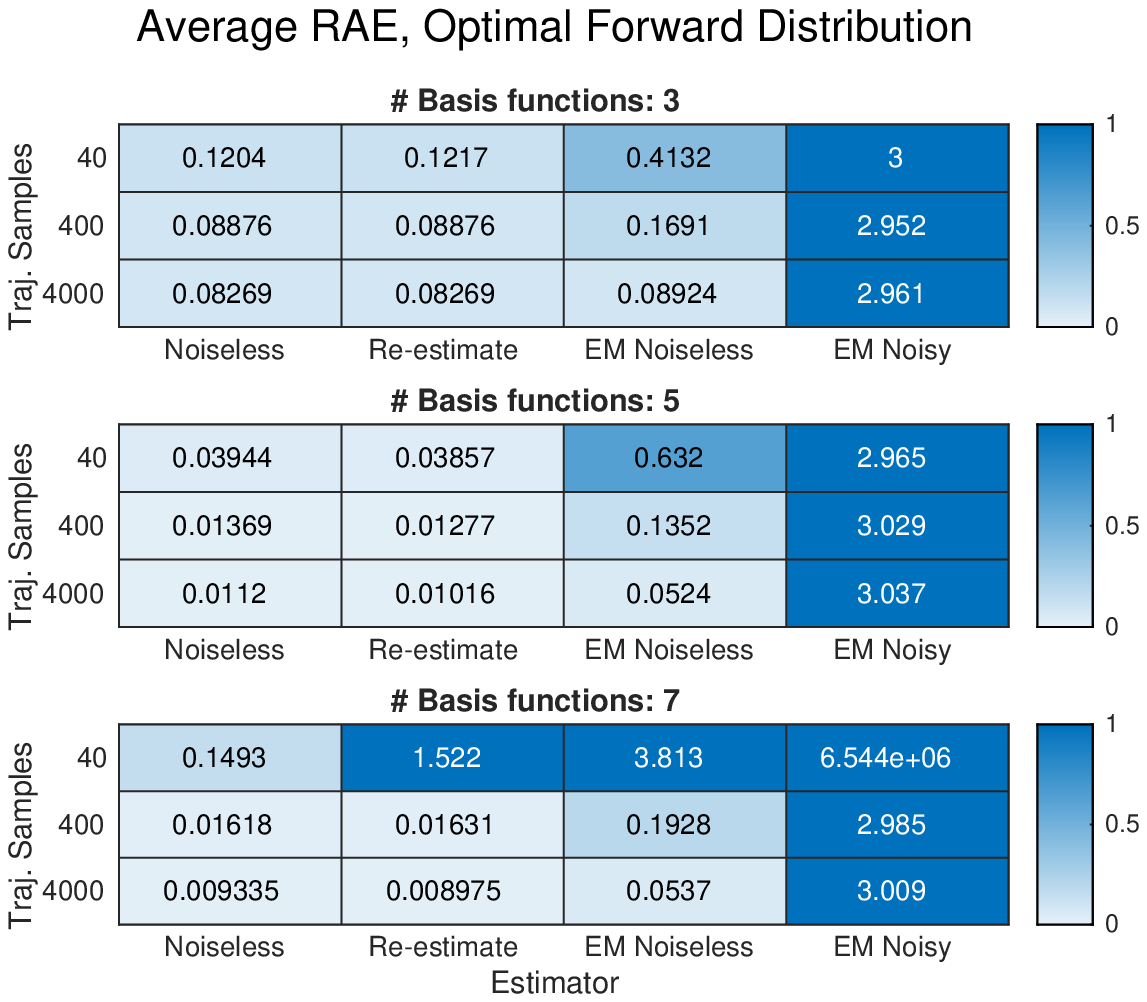}}
    \hspace{2em}
    \subfloat[Suboptimal forward sampling distribution generated with $K^\text{subopt}$ 
        (Off-policy estimators).]
        {\includegraphics[width=0.47\linewidth]{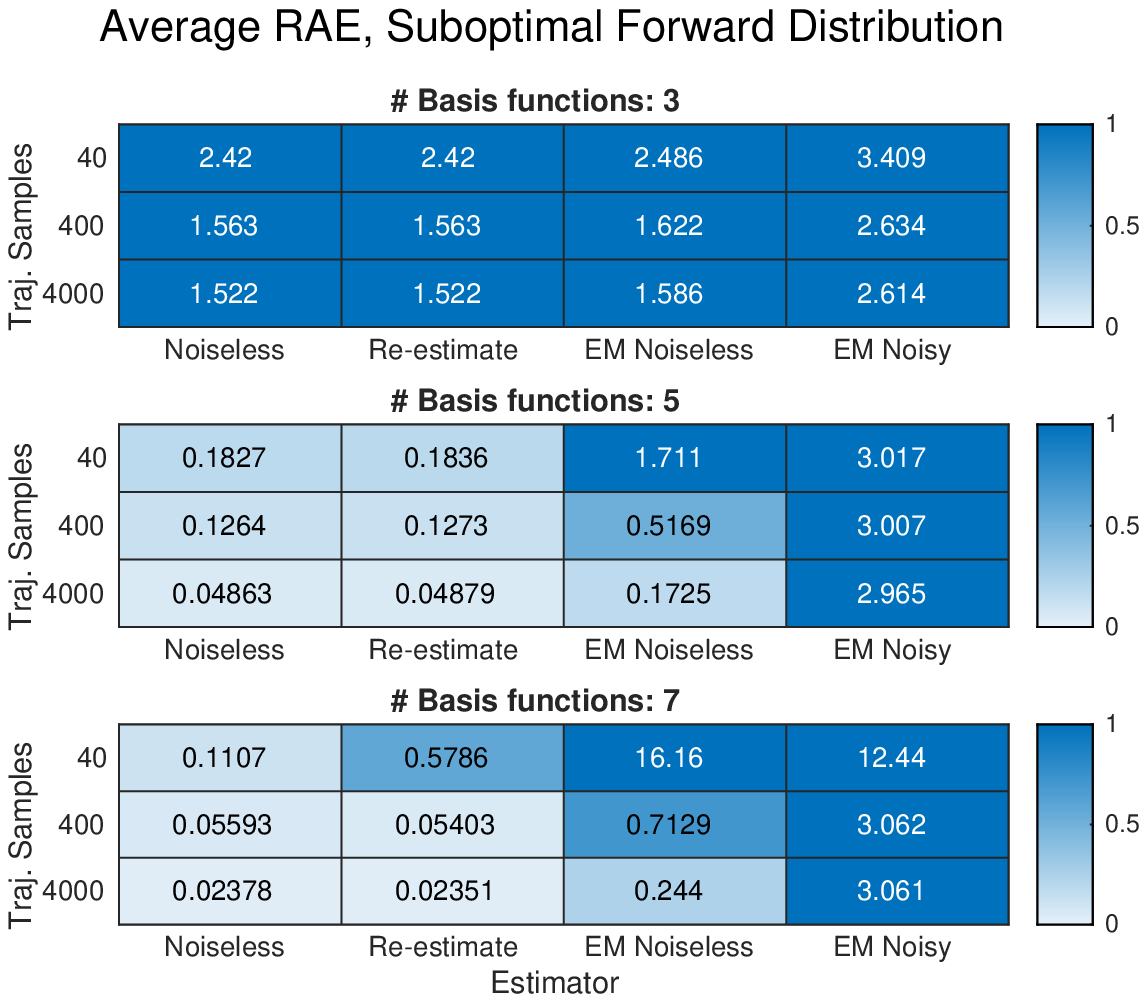}}
        \caption{Heatmaps of experiments comparing the proposed estimators (Noiseless/Re-estimate) 
            against na\"ive estimators (EM Noiseless/EM Noisy), with 
        varying numbers of basis functions 
        and numbers of trajectory samples. 
        Each matrix element is the relative absolute error of the value function
        averaged over both $20$ trials and $N=200$ timesteps.}
    \label{fig:heatmaps}
\end{figure*}


The results show that in all cases the proposed Taylor-based estimators perform as well as the
    Euler-Maruyama estimators and for the vast majority perform significantly better.
Although the Taylor-based estimators generally perform equally well, there are slight differences
    in how they perform in different conditions.
The Taylor-noiseless estimator seems to outperform the re-estimate estimator when the
    number of trajectory samples is low, and vice versa when the number is high.
Recall that the error analysis suggests that the re-estimate estimator has lower bias 
    but higher variance than the Taylor-noiseless estimator.
The simulated results confirm the theoretical results, that is, 
    when the number of trajectory samples is low, high variance makes the re-estimate estimator
    perform poorly, but when there are enough samples to overcome the variance in the
    estimator, the low bias properties can result in better accuracy.
In practice, however, it is likely that the low variance of the Taylor-noiseless estimator is preferable
    to the slightly more bias it introduces.

\subsection{LQR 4D Problem}

We also tested the proposed estimators on a linearized version of the 4-dimensional 
    finite time cart-pole problem,
\begin{align*}
    \dX_s &= \left(\begin{bmatrix}
0 & 1 & 0 & 0\\
0 & a_1 & a_2 & a_3 \\
0 & 0 & 0 & 1\\
0 & a_4 & a_5 & a_6
\end{bmatrix}X_{s}+\begin{bmatrix}
0\\
b_1 \\
0\\
b_2
\end{bmatrix}u_{s}\right) \ds
\\
&\quad +\begin{bmatrix}
0.01 & 0 & 0 & 0\\
0 & 0.1 & 0 & 1\\
0 & 0 & 0.01 & 0\\
0 & 0 & 0 & 0.1
\end{bmatrix} \dW_{s} \\
&= (A X_s + B u_s) \ds + \sigma \dW_s \text{,}
\end{align*}
    where $a_1, a_2, a_3, a_4, a_5, a_6, b_1, b_2$ are constant parameters and
    $x_0 = [0, 0 , \pi/9 , 0]^\top$.
For the suboptimal sampling distribution we selected a discrete time approximation of the time-invariant feedback policy
\begin{align*}
        K^\text{subopt}_s = \bigg ( A  + B
    \begin{bmatrix}
       0 & 0 & k_1 & k_2
    \end{bmatrix} \bigg ) X_s \text{,}
\end{align*} 
    where $k_1, k_2$ are constant parameters.
    The optimal policy is found through the 
    solution of the associated Riccati equations 
    (distributions visualized in Fig.~\ref{fig:cartpole_acc}(a)).

The value function model for $\wt{V}_i$ used Chebyshev functions of degree 2 and lower (15 basis functions).
The RAE approximations \eqref{eq:mad_def} are visualized in Fig.~\ref{fig:cartpole_acc}(b)
    where
    $\wt{\mathcal{C}}_i := \wt{\mathcal{C}}_i^1 \times \wt{\mathcal{C}}_i^2 
    \times \wt{\mathcal{C}}_i^3 \times \wt{\mathcal{C}}_i^4$
    and each $\wt{\mathcal{C}}_i^j$ is defined similarly to \eqref{eq:conf_reg} based on the
    mean and standard deviation of the optimal trajectories in each of the 4 dimensions.

As predicted by the error analysis, since this is an LQR problem and the value function is in the
    class of quadratic functions, the Taylor expansion-based estimators are able to produce
    approximations of the value function with accuracy near machine precision for both conditions.
For the suboptimal forward sampling condition the EM estimators diverge quickly during the backward
    pass.
For the optimal forward sampling condition, corresponding to the on-policy estimation, the
    EM estimators perform mediocre compared to the value function's variance
    and their error is still several orders of magnitudes higher than the Taylor estimators.

These results confirm that the proposed estimators are able to achieve near perfect performance
    on the most common problem in stochastic optimal control.
Further, they confirm that utilizing the second-order derivatives of the value function 
    is crucial for Girsanov-inspired off-policy estimator schemes, contrary to what
    na\"ive application of the theory would suggest.

\begin{figure}
    \centering
    \subfloat[Trajectory distributions for the two sampling conditions 
        ($K_i^\text{optimal}$ / $K_i^\text{subopt}$)]
        {\includegraphics[width=0.92\linewidth]{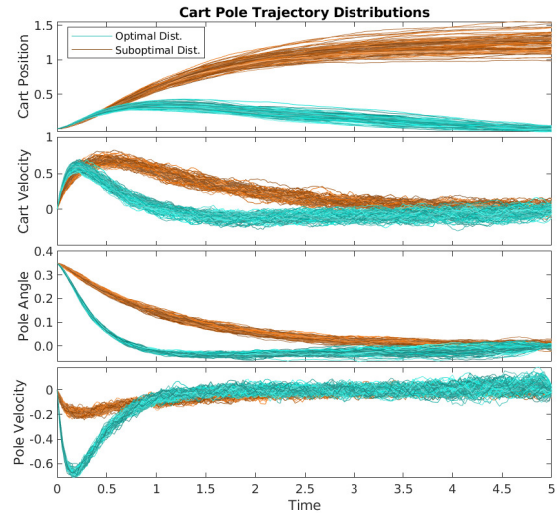}}
        \\
        \subfloat[Relative absolute error \eqref{eq:mad_def} for the two sampling conditions 
        ($K_i^\text{optimal}$ / $K_i^\text{subopt}$)]
        {\includegraphics[width=0.92\linewidth]{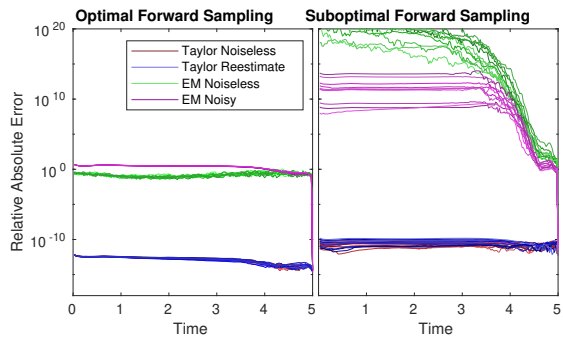}}
    \caption{Comparing the accuracy of the estimators on a 4-dimensional LQR approximation of
        cart-pole balancing system.
    }
    \label{fig:cartpole_acc}
\end{figure}

\section{Conclusion} 


Taylor-based estimators for numerically solving
    Feynman-Kac FBSDEs have been demonstrated to be significantly more accurate than
    na\"ive Euler-Maruyama-based estimators
    through both error analysis and numerical simulation.
These estimators are derived by using higher-order Taylor expansions and
    following the spirit of the continuous-time Feynman-Kac-Girsanov formulation.
Both error analysis and numerical simulation confirm that these estimators have 
    very high accuracy when applied to LQR problems.
Further, in simulation, the proposed estimators are orders of magnitude more accurate than the EM
    estimators in both LQR and nonlinear problems.
Using these results, this paper also proposes a method to use the estimated value function parameters
     for generating an improved policy.

Moving forward, the primary challenge with Feynman-Kac FBSDE methods as presented here is
    how to produce robust iterative methods.
Although value function approximation can be extremely accurate in  the proximity of the initial
    forward pass, even for off-policy methods, Runge's phenomenon begins dominating outside
    the sampling distribution.
As a consequence, when in some extrapolative region,  
    the approximation significantly underestimates the true value function,
    policy improvement begins to fail and future iterations  are constructed based on divergent policies with little room for improvement
    aside from starting over.


To overcome such difficulties, a potential solution is to use a two-phase algorithm where in 
the first phase, following the steps presented in this article, the initial policy is sampled and a single on-policy backward pass
    is performed to find the value function associated with the initial policy.
This work establishes that we can obtain a high-accuracy estimate of the value function in this way.
In the second phase, a gradual optimization
    technique such as stochastic gradient descent can be used to refine the value function 
    and policy approximation.
The arguments in the minimizations \eqref{eq:alphastar} and \eqref{eq:optimprove} need not be
    fully minimized, but can be differentiated with respect to the value function and policy parameters 
    to produce a step of stochastic gradient descent. 
Such techniques are similar to those used in deep reinforcement techniques such as
    deep deterministic policy gradient (DDPG) \cite{Lillicrap2016},
    but when combined with the proposed estimators might accelerate convergence considerably.


\section{Acknowledgements}
This work has been supported by NSF awards CMMI-1662523 and IIS-2008686 
    and ONR award N00014-18-1-2828.
The authors would like to thank Evangelos Theodorou for many helpful discussions and comments.

\bibliographystyle{plain}
\bibliography{library}

\appendix

\section{Proof of Proposition~\ref{cor:odd}} \label{sec:cor_odd}

\begin{proof}
In the following, the variable 
\begin{align*}
    \alpha := (\alpha_1, \ldots, \alpha_n) \in \N^n \text{,}
\end{align*}
    is used as multi-index notation,
\begin{align}
    |\alpha| &:= \alpha_1 + \cdots + \alpha_n, &
    \alpha! &:= \alpha_1! \cdots \alpha_n!, \nonumber \\
    x^\alpha &:= x_1^{\alpha_1} \cdots x_n^{\alpha_n}, &
    \partial_{x_\alpha} &:= 
    \frac{\partial^{|\alpha|}}{\partial x_1^{\alpha_1} \cdots \partial x_n^{\alpha_n}} \text{.} \nonumber
\end{align}
Let $j\geq 3$ be an odd number and suppose $\wt{V}^\mu_{i+1} \in C^k(\R^n)$ for some $k\geq j$.
The $j$-th order term of the Taylor expansion residual is given by Taylor's theorem as
\begin{align}
    \sum_{|\alpha| = j} \frac{1}{\alpha!} 
    \partial_{x_\alpha} \wt{V}^\mu_{i+1}(\ol{X}^\Q_{i+1}) (\Sigma_i W^\Q_i)^\alpha \label{eq:delta3}\text{.} 
\end{align}
It can be shown with algebra and the multinomial theorem 
    that there exists functions $\gamma_\alpha$ such that
    the $W^\Q_i$ terms can be linearly separated from the others,
\begin{align}
    \sum_{|\alpha| = j} 
    \gamma_\alpha(\partial_{x_\alpha} \wt{V}^\mu_{i+1}(\ol{X}^\Q_{i+1}), \Sigma_i) (W^\Q_i)^\alpha \text{.}
    \label{eq:oddterms}
\end{align}
Since both $\partial_{x_\alpha} \wt{V}^\mu_{i+1}(\ol{X}^\Q_{i+1})$ and $\Sigma_i$ are
    $X_i$-measurable, when taking the conditional expectation the operator passes inside
\begin{align}
    \sum_{|\alpha| = j} 
    \gamma_\alpha(\partial_{x_\alpha} \wt{V}^\mu_{i+1}(\ol{X}^\Q_{i+1}), \Sigma_i) 
    \E_{\wt{\Q}}[(W^\Q_i)^\alpha|X_i]
    \text{.} \label{eq:oppassinside}
\end{align}
Due to the independence of the different dimensions of $W^\Q_i$,
    the conditional expectation inside \eqref{eq:oppassinside} can be expanded into the product
\begin{align}
    \E_{\wt{\Q}}[(W^\Q_i)_1^{\alpha_1}|X_i] \cdots
    \E_{\wt{\Q}}[(W^\Q_i)_n^{\alpha_n}|X_i] \text{.}
\end{align}
Since $|\alpha|$ is odd, there exists an $l$ such that $\alpha_l$ is odd.
The properties of the standard normal distribution guarantee
\begin{align}
    \E_{\wt{\Q}}[(W^\Q_i)_l^{\alpha_l}|X_i] = 0 \text{,}
\end{align}
    and thus, \eqref{eq:oppassinside} is zero as well, so we arrive at the result
\begin{align}
    \E_{\wt{\Q}}[\sum_{|\alpha| = j} 
    \gamma_\alpha(\partial_{x_\alpha} 
    \wt{V}^\mu_{i+1}(\ol{X}^\Q_{i+1}), \Sigma_i) (W^\Q_i)^\alpha |X_i] = 0 \text{.}
\end{align}
\end{proof}

\section{Proof of Lemma~\ref{lemma:dtgirsanov}} \label{sec:dtgirsanov}

\begin{proof} 
    We have that $W_i^\Q$ is $\wt{\F}_{i+1}$-measurable because $W_i^\P$ and $D_i$ are measurable by definition.
For the joint distribution $W^\Q_{0:i} := \{W^\Q_j\}_{j=0}^{i}$ to be mutually independent normal
    random variables, its
    density function must be
    $
    \prod_{j=0}^{i} p_{\mathcal{N}}(w^\Q_j) \text{,} \nonumber
    $
    where
\begin{align}
    p_{\mathcal{N}}(w) := \gamma \exp(-\frac{1}{2} \|w\|^2)
    \text{,}
\end{align}
    is the multivariate normal density where $w \in \R^n$ 
    and $\gamma$ is a normalizing constant.
We now verify this is indeed the density of $W^\Q_{0:i}$ by showing
\begin{align}
    &(\wt{\Q}_{i+1} \circ (W^\Q_{0:i})^{-1})(A_{0:i}) \nonumber \\
    &\quad = \int_{B}\, \mathrm{d} \wt{\Q}_{i+1}  
    = \int_{A_{0:i}} \prod_{j=0}^{i} p_{\mathcal{N}}(w^\Q_j) \, \mathrm{d} w^\Q_{0:i} 
    \label{eq:densitydef}
    \text{,}
\end{align}
    where $A_{0:i} = A_0 \times \cdots \times A_i$, the product of sets in
    the sigma algebras $A_j \in \mathcal{B}(\R^n)$    
    and $B = (W^\Q_{0:i})^{-1}(A_{0:i}) \in \wt{\F}_{i+1}$.
\footnote{We prove this only for the rectangle semi-algebra over sets constructed as $A_{0:i}$ 
        \cite[p.~144]{resnick2003probability},
    and assume that the filtration is constructed such that $B' \in \wt{\F}_{i+1}$
    for all $A'_{0:i} \in \mathcal{B}(A_0 \times \ldots \times A_i)$.
    The probability measure \eqref{eq:densitydef} over $A_{0:i}$ is clearly $\sigma$-additive
    \cite[p.~43]{resnick2003probability} so we can use an extension theorem 
    \cite[p.~48]{resnick2003probability} to show
    that there exists a probability measure that extends to all $A'_{0:i}$.
}

Notice that
\begin{align}
    \Theta_{i+1} &= \prod_{j=0}^i \varphi(D_j,W_j^\P) \text{,} \nonumber \\
    \varphi(d,w) &:= \exp \bigg (-\frac{1}{2} \| d \|^2 + d^\top w \bigg) \nonumber \text{,}
\end{align}
    and let
    $\ol{\Theta}_{i+1} := \prod_{j=0}^i \ol{\varphi}_j(D_j,W^\P_j)$,
    where
    $\ol{\varphi}_j(D_j,W^\P_j) := \textbf{1}_{A_j}(W^\P_j - D_j) \varphi(D_j,W^\P_j)$,
    be the restriction of $\Theta_{i+1}$ to the set $B$.
By \eqref{eq:dtmeas} we have
\begin{align}
    \int_B \mathrm{d} \wt{\Q}_{i+1} 
    &= \int_B \Theta_{i+1} \mathrm{d} \wt{\P}_{i+1} 
    = \int_{\wt{\Omega}} \ol{\Theta}_{i+1} \mathrm{d} \wt{\P}_{i+1} \text{,} \nonumber
\end{align}
    and by the properties of conditional expectation \cite[p.~10]{yong1999stochastic} we have
\begin{align}
    &= \int_C \E_{\wt{\P}}[\ol{\Theta}_{i+1} | d_{0:i}, w^\P_{0:i-1}] 
    \mathrm{d} \wt{\P}_{D_{0:i}, W^\P_{0:i-1}} \nonumber \\
    &= \int_C \ol{\Theta}_{i} \E_{\wt{\P}}[\ol{\varphi}_i(d_i,W^\P_i) | d_{0:i}, w^\P_{0:i-1}] 
    \mathrm{d} \wt{\P}_{D_{0:i}, W^\P_{0:i-1}} \nonumber 
    \text{,}
\end{align}
    where
\begin{align}
    \mathrm{d} \wt{\P}_{D_{0:i}, W^\P_{0:i-1}} &= (\wt{\P} \circ (D_{0:i}, W^\P_{0:i-1})^{-1})
    (\mathrm{d} d_{0:i}, \mathrm{d} w^\P_{0:i-1}) \nonumber \text{.}
\end{align}
    is the derivative with respect to the distribution over the variables $(D_{0:i}, W^\P_{0:i-1})$,
    the conditional expectation is with respect to these variables' values given,
    and $C$ is the image $C = (D_{0:i}, W^\P_{0:i-1})(\wt{\Omega})$.
    The value $\ol{\Theta}_i$ can be taken outside the conditional expectation because it is
    deterministic with respect to the conditioned variables $(d_{0:i-1}, w^\P_{0:i-1})$.
From \cite[Chapter~1, Proposition~1.10]{yong1999stochastic}, it follows that
\begin{align}
    &\E_{\wt{\P}}[\varphi(d_i,W^\P_i) | d_{0:i}, w^\P_{0:i-1}] \nonumber \\
    &\quad = \int_{\wt{\Omega}} \ol{\varphi}_i(d_i,W^\P_i(\omega)) 
    \, \p(\mathrm{d} \omega | d_{0:i}, w^\P_{0:i-1}) \nonumber 
    \text{,}
\end{align}
    where $\p$ is a regular conditional probability.
By a change-of-variable and using the independence assumption on $W^\P_i$, we can write this
    integral as
\begin{align}
    & \int_{\R^n} \varphi(d_i,w^\P_i) \textbf{1}_{A_i}(w^\P_i - d_i) 
    \, \wt{\P}_{W^\P_i} (\mathrm{d} w^\P_i) \nonumber 
    \text{,}
\end{align}
    where $\wt{\P}_{W^\P_i}$ is the distribution in $\wt{\P}$ over $W^\P_i$ alone.
Converting the distribution to a density gives
\begin{align}
    & \int_{\R^n} \varphi(d_i,w^\P_i) \textbf{1}_{A_i}(w^\P_i - d_i) 
    p_{\mathcal{N}} (w^\P_i) \, \mathrm{d} w^\P_i \nonumber 
    \text{,}
\end{align}
    which can be algebraically reduced to
\begin{align}
    & \int_{\R^n} 
    p_{\mathcal{N}} (w^\P_i - d_i) \textbf{1}_{A_i}(w^\P_i - d_i) \, \mathrm{d} w^\P_i \nonumber 
    \text{.}
\end{align}
Define the change-of-variable mapping 
    $w^\Q_i|_{d_i}(w^\P_i) = w^\P_i - d_i$,
    noting that the Lebesgue measures will be equivalent due to translation invariance
    $\mathrm{d} w^\Q_i|_{d_i} = \mathrm{d} w^\P_i$,
    and apply it to the integral, giving
\begin{align}
    & \int_{A_i} 
    p_{\mathcal{N}} (w^\Q_i|_{d_i}) \, \mathrm{d} w^\Q_i|_{d_i} \nonumber 
    \text{.}
\end{align}
The value of this integral does not change for any $d_i$ in the range of $D_i$.
Plugging this back into the original integral and pulling it out due to its invariance,
    we have
\begin{align}
    \int_B \mathrm{d} \wt{\Q}_{i+1} 
    &= \int_C \Theta_{i} \bigg ( \int_{A_i} 
    p_{\mathcal{N}} (w^\Q_i|_{d_i}) \, \mathrm{d} w^\Q_i|_{d_i} \bigg )
    \mathrm{d} \wt{\P}_{D_{0:i}, W^\P_{0:i-1}} \nonumber \\
    &=\int_{A_i} 
    p_{\mathcal{N}} (w^\Q_i) \, \mathrm{d} w^\Q_i 
    \int_C \Theta_{i} \, 
    \mathrm{d} \wt{\P}_{D_{0:i}, W^\P_{0:i-1}} \nonumber \\
    &=\int_{A_i} 
    p_{\mathcal{N}} (w^\Q_i) \, \mathrm{d} w^\Q_i 
    \int_C \Theta_{i} \, 
    \mathrm{d} \wt{\P}_{i+1} \nonumber 
    \text{.}
\end{align}
Conditioning the right integral on the variables $(D_{0:i-1}, W^\P_{0:i-2})$
    we can continue this process until the right integral disappears and
    we are left with the right hand side of \eqref{eq:densitydef}.
\end{proof}

\section{Proof of Theorem~\ref{thm:biasvarhat} \& Theorem~\ref{thm:newests}}
\label{sec:ProofThm32n33}

\begin{proof}
Using the result \eqref{eq:deltadelta2} of Lemma~\ref{thm:driftdeltaest} we have
\begin{align}
    \wh{Y}_i &:= \wh{Y}_{i+1} - \Delta \wh{Y}_i \nonumber \\
    &= \wh{Y}_{i+1} - \Delta Y_i + (\delta^{\Delta \wh{Y}}_{i+1}
    - \E_{\wt{\Q}}[\delta^{\Delta \wh{Y}}_{i+1} |X_i, K_i]) \nonumber 
    \text{,}
\end{align}
    and so the expression for the bias is 
\begin{align}
    \E_{\wt{\P}}[Y_i - \wh{Y}_i | X_i, K_i]
    &= \E_{\wt{\P}}[Y_{i+1} - \wh{Y}_{i+1} | X_i, K_i] - \eps^{\P | \Q}_{i+1}
    \text{.}  \nonumber 
\end{align}
The variance of the estimator is
\begin{align}
    \var_{\wt{\P}}[\wh{Y}_i | X_i, K_i]
    &= \var_{\wt{\P}}[\wh{Y}_{i+1} - \Delta Y_i \nonumber \\
    &\quad+ (\delta^{\Delta \wh{Y}}_{i+1}
    - \E_{\wt{\Q}}[\delta^{\Delta \wh{Y}}_{i+1} | X_i, K_i]) | X_i, K_i] \nonumber \\
    &= \var_{\wt{\P}}[\delta^{\Delta \wh{Y}}_{i+1} - (Y_{i+1} - \wh{Y}_{i+1})| X_i, K_i] 
    \text{,}  \nonumber 
\end{align}
    noting that we can drop the terms $Y_i$ and $\E_{\wt{\Q}}[\delta^{\Delta \wh{Y}}_{i+1} |X_i, K_i]$
    because they are $(X_i,K_i)$-measurable.

For the re-estimate estimator we have
\begin{align}
    Y_{i+1} - \wh{Y}^{\textnormal{re-est}}_{i+1} 
    &= V^\mu_{i+1}(X_{i+1}) - \wt{V}^\mu_{i+1}(X_{i+1}) \nonumber \\
    &= \delta^{\wt{V}}_{i+1}
    \text{,} 
\end{align}
    and for the noiseless estimator we have
\begin{align}
    Y_{i+1} - \wh{Y}^{\textnormal{noiseless}}_{i+1} 
    &= V^\mu_{i+1}(X_{i+1}) - \wt{Y}_{i+1} \nonumber \\
    &= V^\mu_{i+1}(X_{i+1}) - (\wt{V}^\mu_{i+1}(X_{i+1}) - \delta^\textnormal{h.o.t.}_{i+1}) \nonumber \\
    &= \delta^{\Delta \wh{Y}}_{i+1}
    \text{,} 
\end{align}
    due to \eqref{eq:errwtY}.
Plugging these two equalities into the general expressions for the bias and variance and
    doing simple reductions yields the theorem results.
Note that Theorem~\ref{thm:biasvarhat} is proved by setting $\wt{\P} \equiv \wt{\Q}$,
    which entails $\eps^{\P | \Q}_{i+1} \equiv 0$, and by excluding $K_i$ from the conditional
    expectations.
\end{proof}

\section{Proof of Theorem~\ref{thm:diffmeaserr}} \label{sec:diffmeaserr}

\begin{proof}
    First note that the process $\{\Theta_i\}$ is a martingale in $\wt{P}$, that is,
    $\E_{\wt{\P}}[\Theta_j|\F_i] = \Theta_i$ for $j \geq i$.
    This can be shown by defining the measures 
\begin{align} 
    \mathrm{d} \wt{\Q} &= \Theta_j \mathrm{d} \wt{\P}_j 
    = \frac{\Theta_j}{\Theta_i} \mathrm{d} \wt{\pR}_i  \nonumber
    \text{,}
\end{align} 
    and noting that, for all $B_i \in \F_i$,
\begin{align} 
    \int_{B_i} \Theta_i \mathrm{d} \wt{\P}_j 
    &= \int_{B_i} \mathrm{d} \wt{\pR}_i 
    = \int_{B_i} \mathrm{d} \wt{\Q}
    = \int_{B_i} \Theta_j \mathrm{d} \wt{\P}_j \text{,} \nonumber
\end{align} 
    where the inner equality is due to $\wt{\pR}_i$ and $\wt{\Q}$ agreeing on $B_i$.

From \cite[Lemma~1]{lowther_2011} we have
\begin{align}
    \E_{\wt{\Q}}[\delta^{\Delta \wh{Y}}_{i+1} |\F_{i}] 
    &= \frac{\E_{\wt{\P}}[\Theta_{N} \delta^{\Delta \wh{Y}}_{i+1} |\F_{i}]}
            {\E_{\wt{\P}}[\Theta_{N} |\F_{i}]} \nonumber \\
    &= \frac{\Theta_{i} \E_{\wt{\P}}[\Theta_{i}^{-1} \Theta_{N} \delta^{\Delta \wh{Y}}_{i+1} |\F_{i}]}
            {\Theta_{i}} \nonumber \\
    &= \E_{\wt{\P}}[ (\Theta_{N} \Theta_{i+1}^{-1}) \varphi(D_i,W^\P_i) 
    \delta^{\Delta \wh{Y}}_{i+1} |\F_{i}]
         \nonumber \\
    &= \E_{\wt{\P}}[ \varphi(D_i,W^\P_i) 
    \delta^{\Delta \wh{Y}}_{i+1} |\F_{i}]
         \nonumber 
        \text{,}
\end{align}
    where the final equality is due to the tower property of conditional expectation and the
    $\F_{i+1}$ measurability of the remaining terms.
Again by the tower property of conditional expectation we have
\begin{align}
    \E_{\wt{\Q}}[\delta^{\Delta \wh{Y}}_{i+1} |X_i, K_i] 
    &= \E_{\wt{\P}}[\varphi(D_i,W^\P_i) \delta^{\Delta \wh{Y}}_{i+1} |X_i, K_i] \text{.}  \nonumber
\end{align}
By the Cauchy-Schwartz inequality, we have that
\begin{align}
    &| \E_{\wt{\Q}}[\delta^{\Delta \wh{Y}}_{i+1} |X_i, K_i] | \nonumber \\
    &\quad \leq \E_{\wt{\P}}[\varphi(D_i,W^\P_i)^2|X_i, K_i]^{1/2} 
    \E_{\wt{\P}}[(\delta^{\Delta \wh{Y}}_{i+1})^2 |X_i, K_i]^{1/2} \text{.} \nonumber
\end{align}
Using properties of log-normal distributions \cite{crow1987lognormal} we have
\begin{align}
    \E_{\wt{\P}}[\varphi(D_i,W^\P_i)^2|X_i, K_i]
    &= \E_{\wt{\P}}[\exp(\| D_i \|^2)|X_i, K_i]  \nonumber \\
    &= \exp(\| D_i \|^2) \text{,} \nonumber
\end{align}
    which, upon substitution, yields the desired result.
\end{proof}

\end{document}